\documentclass{pset}
\usepackage{enumerate}
\title{NIP and Distal Metric Structures}
\author{Aaron Anderson}

\begin{document}

\maketitle

\begin{abstract}
    Model theory, machine learning, and combinatorics each have generalizations of VC-dimension for fuzzy and real-valued versions of set systems.
    These different dimensions define a unique notion of a VC-class for both fuzzy sets and real-valued functions.
    We study these VC-classes, obtaining generalizations of certain combinatorial results from the discrete case.
    These include appropriate generalizations of $\varepsilon$-nets, the fractional Helly property and the $(p,q)$-theorem.
    
    We then apply these results to continuous logic.
    We prove that NIP for metric structures is equivalent to an appropriate generalization of honest definitions, which we use to study externally definable predicates and the Shelah expansion.
    We then examine distal metric structures, providing several equivalent characterizations,
    in terms of indiscernible sequences, distal types, strong honest definitions, and distal cell decompositions.
\end{abstract}
\maketitle
\tableofcontents

\section{Introduction}

Distal structures were first studied as a way to characterise non-stable behavior in NIP theories, and defined in terms of indiscernible sequences\cite{distal_simon}.
They include some important non-stable NIP structures, such as weakly $o$-minimal structures and the $p$-adics.
Subsequently, distality was re-defined combinatorially, in terms of strong honest definitions or distal cell decompositions,
generalizing $o$-minimal cell decompositions, and providing the most general model-theoretic setting for semialgebraic incidence combinatorics\cite{cs2}\cite{cgs}\cite{distal_reg}.

Continuous logic replaces the standard first-order structures of model theory with metric structures, and formulas with continous functions to the real interval $[0,1]$\cite{mtfms}.
This makes it the natural setting to study analytic objects such as probability algebras, Banach spaces, and $C^*$-algebras.
It also has natural connections to topological dynamics, as Polish groups are exactly the automorphism groups of metric structures, and new research has linked stability and NIP to dynamics in this way\cite{melleray}\cite{byt}\cite{ibarlucia1}.
Stability, NIP, $n$-dependence, and some other dividing lines of neostability theory have already been defined for metric structures,
with applications such as continuous $n$-dependent or stable regularity\cite{local_stab}\cite{randomVC}\cite{ct}\cite{csr}.
Meanwhile, distal metric structures have only been mentioned in the context of hyperimaginaries\cite{kp21}.

In this paper, we aim to lay the groundwork for studying distality in continuous logic.
We set up the basic theory of distal metric structures, proposing continuous versions of several definitions of distality, and proving them equivalent.
Along the way, we prove results relevant to all NIP metric structures, including versions of honest definitions and uniform definability of types over finite sets (UDTFS),
generalizing results from \cite{cs1} and \cite{cs2} in the discrete case.
Forthcoming papers\cite{anderson2}\cite{anderson3} will provide examples of distal metric structures and consider distal regularity (as developed in \cite{distal_reg} and simplified in \cite{simon_distal_reg}) in the context of continuous logic,
providing further characterizations of distal metric structures in terms of Keisler measures.

In order to understand distality in continuous logic, we must first better understand NIP, and the various fuzzy and real-valued generalizations of VC-dimension.
In Section \ref{sec_combo}, we use these to prove real-valued versions of some classic combinatorial theorems of VC-classes.
Classically, a set system on a set $X$ is a set or family of subsets of $X$, which can be thought of in terms of their characteristic functions $X \to \{0,1\}$.
A fuzzy set system replaces the characteristic function $X \to \{0,1\}$ with a characteristic function $X \to \{0,1,*\}$, where $*$ denotes an indeterminate truth value.
The most important examples of fuzzy set systems come from classes of functions $X \to [0,1]$.
Given any family $\mathcal{F}$ of such functions, and any $0 \leq r < s \leq 1$, we can define a fuzzy set system by replacing each function
$f : X \to [0,1]$ with the function $f_{r,s}$ such that $f(t) = 0$ for $t \leq r$, $f(t) = *$ for $r < t< s$, and $f(t) = 1$ for $t \geq s$.
These fuzzy set systems arising from real-valued functions are central to Ben Yaacov's development of NIP in continuous logic, including a proof that randomizations of NIP structures are NIP\cite{randomVC}.

In Section \ref{sec_combo}, we review the different notions of VC-dimension for fuzzy set systems and real-valued function systems,
such as fuzzy VC-dimension, Rademacher complexity, and covering numbers, and compare these, checking that all of these give rise to the same definition of a VC class of functions.
We then show that VC classes of fuzzy sets admit $\varepsilon$-nets, while VC classes of functions admit $\varepsilon$-approximations.
We then use a combination of these techniques to show a fractional Helly theorem (Theorem \ref{thm_helly}) and a real-valued $(p,q)$-theorem (Theorem \ref{thm_pq}),
which we will later apply to the model theoretic context to get uniform (strong) honest definitions.

In Section \ref{sec_HD}, we apply the results of Section \ref{sec_combo} to NIP metric structures, using background on continuous logic provided in Section \ref{sec_prereq}.
Just as in classical logic, where an NIP structure is one where every definable class of sets has finite VC-dimension, a metric structure is NIP (as defined in \cite{randomVC}) when every definable class of functions is a VC class in any of the equivalent senses of Section \ref{sec_combo}.
We find several more equivalent definitions of NIP, summarized in Theorem \ref{thm_NIP_equiv}.
In particular, NIP metric structures are characterized by the following version of honest definitions:
\begin{thm}\label{thm_intro_HD}
    Let $A$ be a closed subset of $M^x$ where $M \preceq \mathcal{U}$ and $(M,A)\preceq (M',A')$.
    Let $\phi(x;y)$ be a definable predicate.
    Then there exists a definable predicate $\psi(x;z)$, which we call a \emph{uniform honest definition} for $\phi(x;y)$,
    such that for every $b \in M^y$, there exists $d \in A^z$ such that
    \begin{itemize}
        \item for all $a \in A$, $\phi(a;b) = \psi(a;d)$
        \item for all $a' \in A'$, $\phi(a';b) \leq \psi(a';d)$.
    \end{itemize}
\end{thm}
We then define the Shelah expansion of a metric structure by externally definable predicates, and use honest definitions to show that the Shelah expansion of an NIP metric structure is NIP, just as in the classical case developed in \cite{cs1}.

With these techniques for studying NIP metric structures, we turn our attention to distal metric structures in Section \ref{sec_distal}.
Distal metric structures were briefly mentioned in \cite{kp21}, defined by applying the definition of distal indiscernible sequences to the continuous logic context.
In this section, we flesh out the theory of distal metric structures, starting with that indiscernible sequence definition, and proving several equivalent characterizations:
\begin{thm}[Theorem \ref{thm_distal}]
    If a metric theory $T$ is NIP, then the following are equivalent:
    \begin{enumerate}
        \item $T$ is distal.
        \item Every global type is distal.
        \item Every formula admits strong honest definitions (see Definition \ref{defn_SHD}).
        \item Every formula admits an $\varepsilon$-distal cell decomposition for each $\varepsilon > 0$ (see Definition \ref{defn_DCD}).
    \end{enumerate}
\end{thm}
This generalizes characterizations of distality from \cite{nip_guide}, \cite{cs2}, and \cite{cgs} to work with metric structures.

\subsection*{Acknowledgements}
We thank Artem Chernikov for advising and support throughout this project, Ita\"i Ben Yaacov for advising while the author was in Lyon, and James Hanson for several helpful ideas and conversations about continuous logic.
The author was partially supported by the Chateaubriand fellowship, the UCLA Logic Center, the UCLA Dissertation Year Fellowship,
and NSF grants DMS-1651321 and DMS-2246598
during the writing process.

\section{Fuzzy Combinatorics}\label{sec_combo}
In this section, we will generalize some combinatorial facts about set systems and relations of finite VC-dimension to fuzzy set systems and fuzzy relations.

The VC-dimension of fuzzy set systems was introduced for model theory purposes in \cite{randomVC}, and for machine learning purposes in \cite{PCC}.
A fuzzy subset $S$ of a set $X$, denoted $S \sqsubseteq X$, is formalized as pair $(S_+, S_-)$ of disjoint subsets of $X$, where $S_+$ is the set of elements such that $x \in S$, $S_-$ is the set of elements such that $x \not \in S$, but for $x \in X \setminus (S_+ \cup S_-)$, the truth value of $x \in S$ is undefined. (These can also be modeled as partial functions to $\{0,1\}$ on $X$, or as in \cite{PCC}, functions to $\{0,*,1\}$.)
A fuzzy set system on $X$ is a set of fuzzy subsets of $X$, and a fuzzy relation between $X$ and $Y$ is a fuzzy subset of $X \times Y$.
A fuzzy relation $R \sqsubseteq X \times Y$ can produce two fuzzy set systems: $R^Y$ is the fuzzy set system on $X$ given by $\{(\{x : (x,y) \in R_+\},\{x : (x,y) \in R_-\}) : y \in Y\}$, and $R_X$ is the similarly-defined fuzzy set system on $Y$.
Each fuzzy relation $R \sqsubseteq X \times Y$ has a corresponding dual fuzzy relation, $R^* \sqsubseteq Y \times X$, given by $(y,x) \in R_+^* \iff (y,x) \in R_+$ and $(y,x) \in R_-^* \iff (y,x) \in R_-$.
Any fuzzy set system $\mathcal{F}$ can also be thought of as a fuzzy relation $X \times \mathcal{F}$, given by $(\{(x,S): x \in S_+\},\{(x,S): x \in S_-\})$, and thus we can define a dual fuzzy set system, $\mathcal{F}^*$, which is the fuzzy set system on $\mathcal{F}$ induced by the dual of that fuzzy relation.

Sometimes, for combinatorial results, it is more useful to think of a fuzzy subset of $X$ as a pair of nested subsets, as $S_+ \subseteq X \setminus S_-$, where we think of the inner subset as the elements that are definitely in $S$, and the outer subset as the elements that could possibly be in $S$. If $\mathcal{F}$ is a fuzzy set system on $X$, then we can define the \emph{inner} and \emph{outer} set systems by $\mathcal{F}_i = \{S_+ : S \in \mathcal{F}\}$ and $\mathcal{F}_o = \{X \setminus S_- : S \in \mathcal{F}\}$. We can translate many of the combinatorial theorems known for non-fuzzy set systems by showing that if the assumptions of the theorem hold for $\mathcal{F}_i$, then the results will hold for $\mathcal{F}_o$.

\begin{defn}
Let $\mathcal{F}$ be a fuzzy set system on $X$ and $Y \subseteq X$. We will define the basic notions of shattering and the shatter functions associated to $\mathcal{F}$.

\begin{itemize}
    \item Let $\mathcal{F} \cap Y$ be the set of all subsets $Z \subseteq Y$ such that there exists $S \in \mathcal{F}$ with $S_+ \cap Y = Z$ and $S_- \cap Y = Y \setminus Z$.
    \item Let $\pi_{\mathcal{F}}(n) = \max_{Y \subseteq X: |Y| = n}|\mathcal{F} \cap Y|$.  We call $\pi_{\mathcal{F}}$ the \emph{shatter function} of $\mathcal{F}$.
    \item Say that $\mathcal{F}$ \emph{shatters} $Y$ when $\mathcal{F} \cap Y = \mathcal{P}(Y)$, or equivalently, $|\mathcal{F} \cap Y| = 2^{|Y|}$.
    \item Define the \emph{dual shatter function}, $\pi^*_{\mathcal{F}}$, to be $\pi_{\mathcal{F}^*}$.
\end{itemize} 
\end{defn}

We now have the nomenclature to define the VC-dimension of a fuzzy set system, and VC classes.
\begin{defn}
Let $\mathcal{F}$ be a fuzzy set system on $X$, and $d \in \N$. We say that $\mathcal{F}$ has VC-dimension at least $d$ when $\pi_{\mathcal{F}}(n) = 2^n$ for all $n \leq d$. The \emph{VC-dimension} of $\mathcal{F}$, denoted $\mathrm{vc}(\mathcal{F})$, is then the largest such $d$ if there is one, and is otherwise $\infty$. We say that $\mathcal{F}$ is a \emph{VC class} when $\mathcal{F}$ has finite VC-dimension.

We define the \emph{dual VC-dimension} $\mathrm{vc}^*(\mathcal{F})$ to be the VC-dimension of $\mathcal{F}^*$.
\end{defn}

Note that this notion of dimension differs by 1 (in the finite case) from the notion of VC-index discussed in \cite{randomVC}. This more closely matches the convention adopted in the combinatorics literature that will be cited later.

The following lemma shows that we do not need to define dual-VC classes, as they are the same as VC classes.
\begin{fact}[{\cite[Fact 2.14]{randomVC}}]
If $R \sqsubseteq X \times Y$ is a fuzzy relation, then $R_X$ is a VC class if and only if $R^Y$ is. Thus we can simply speak of VC-relations without specifying whether we are referring to $R_X$ or $R^Y$ having finite VC-dimension.
\end{fact}

In order to understand the shatter function, we note that the Sauer-Shelah lemma translates easily to the fuzzy context:
\begin{fact}[{\cite{randomVC}}]
If $\mathcal{F}$ is a fuzzy set system on $X$ with VC-dimension at most $d$, then for all $n$, $\pi_{\mathcal{F}}(n) \leq p_d(n)$, where $p_d(n) = \sum_{k \leq d} n^k = O(n^d)$.
\end{fact}

Unfortunately, this polynomial bound does not suffice to translate all probabilistic arguments using the shatter function, as for a fuzzy set system $\mathcal{F}$ on $X$ and a subset $Y \subseteq X$, the number of actual possible fuzzy subset intersections $(S_+ \cap Y, S_- \cap Y)$ for $S \in \mathcal{F}$ could be much larger. In some cases, counting a \emph{strong disambiguation} (as described in \cite{PCC}) will be more helpful:

\begin{defn}
If $S \sqsubseteq X$, say that a subset $S' \subseteq X$ \emph{strongly disambiguates} $S$ when $S_+ \subseteq S'$ and $S' \cap S_- = \emptyset$.
Say that a set system $\mathcal{F}'$ on a set $X$ \emph{strongly disambiguates} a fuzzy set system $\mathcal{F}$ on $X$ when for every fuzzy set $S \in \mathcal{F}$, there is some $S' \in \mathcal{F}'$ refining $S$.
\end{defn}

The following lemma is an immediate consequence of \cite[Theorem 13]{PCC}, and can be thought of as a version of Sauer-Shelah for strong disambiguations, though its bound is slightly worse than polynomial.
\begin{lem}\label{lem_disamb}
Let $\mathcal{F}$ be a fuzzy set system on a finite set $X$ of VC index at most $d$. Then there exists a non-fuzzy set system $\mathcal{F}'$ strongly disambiguating $\mathcal{F}$, with $|\mathcal{F}'| \leq |X|^{O(d\log(|X|))}$.
\end{lem}

We now look at fuzzy set systems derived from classes of real-valued functions. If $Q \subseteq [0,1]^X$, and $0 \leq r < s \leq 1$, then $Q$ gives rise to the fuzzy set system $Q_{r,s}$ consisting of the fuzzy sets $q_{r,s} = (q_{\leq r}, q_{\geq s})$ for $q \in Q$, where $q_{\leq r} = \{x : q(x) \leq r\}$ and $q_{\geq s} = \{x : q(x) \geq s\}$.
Then the inner set system of $Q_{r,s}$ is $Q_{\leq r} := \{\{x : q(x) \leq r\}: q \in Q\}$, and the outer is $Q_{< s} := \{\{x : q(x) < s\}: q \in Q\}$. If instead of a set of functions $Q \subseteq [0,1]^X$, we have a function $Q : X \times Y \to [0,1]$, we can define a fuzzy relation $Q_{r,s}$ on $X$ and $Y$.

\begin{defn}
Let $Q \subseteq [0,1]^X$ be a collection of functions.
We say that $Q$ is a \emph{VC-class} when for any $0 \leq r < s \leq 1$, the fuzzy set system $Q_{r,s}$ has finite VC-dimension.

If instead $Q$ is a function $Q : X \times Y \to [0,1]$, we say that $Q$ is a \emph{VC-function} when for any $0 \leq r < s \leq 1$, the fuzzy relation $Q_{r,s}$ has finite VC-dimension.
\end{defn}

\subsection{Rademacher/Gaussian Complexity and $\varepsilon$-Approximations}

In order to express another equivalent definition of VC classes of functions, we need to introduce the concepts of Rademacher/Gaussian complexity and mean width. This definition of a VC class will then allow us to retrieve a version of the VC Theorem, guaranteeing the existence of $\varepsilon$-approximations to VC classes.

\begin{defn}
Let $A \subseteq \R^n$. Let $\sigma$ be a randomly chosen vector in $\R^n$. Define the \emph{mean width} of $A$, $w(A,\sigma)$, to be $\mathbb{E}_\sigma[\sup_{a \in A} \sigma \cdot a]$. 

If $\sigma$ is chosen uniformly from $\{+1,-1\}$, then we call $w(A,\sigma)$ the \emph{Rademacher mean width}, denoted $w_R(A) = w(A,\sigma)$.

If $\sigma = (\sigma_1,\dots,\sigma_n)$, where the $\sigma_i$s are independent Gaussian variables with distribution $N(0,1)$, then we call $w(A,\sigma)$ the \emph{Gaussian mean width}, denoted $w_G(A) = w(A,\sigma)$.
\end{defn}

The following fact allows us to translate between statements using Rademacher and Gaussian variables:

\begin{fact}[{\cite[Exercise 5.5]{wainwright}}]\label{rad_gauss}
For any $A \subseteq [0,1]^n$,
$$w_R(A)\leq \sqrt{\frac{\pi}{2}}w_G(A) \leq 2 \sqrt{\log n}w_R(A)$$
\end{fact}

We can now apply these definitions to function classes.

\begin{defn}
Let $Q \subseteq [0,1]^X$ be a function class, and let $\bar x = (x_1,\dots, x_n) \in X^n$. Define $Q(\bar x) = \{(q(x_1),\dots,q(x_n)): q \in Q\}$. 

Then define the Rademacher mean width $r_Q(n)$ to be $\sup_{\bar x \in X}w_R(Q(\bar x))$, and the Gaussian mean width $g_Q(n)$ to be $\sup_{\bar x \in X}w_G(Q(\bar x))$.

If $\mu$ is a probability measure on $X$, then define the Rademacher complexity $r_{Q,\mu}(n)$ to be $\frac{1}{n}\mathbb{E}_{\mu^n}[w_R(Q(\bar x))]$, and the Gaussian complexity $g_{Q,\mu}(n)$ to be $\frac{1}{n}\mathbb{E}_{\mu^n}[w_G(Q(\bar x))]$. (Note the normalization factor $\frac{1}{n}$ - this is more useful for probability applications.)
\end{defn}

It is easy to see that for all choices of $Q,n,\mu$, we have $\frac{r_Q(n)}{n} \leq r_{Q,\mu}(n)$ and $\frac{g_Q(n)}{n} \leq g_{Q,\mu}(n)$. We now have the language to connect these notions to VC classes:
\begin{lem}
Let $X$ be a set, and $Q \subseteq [0,1]^X$. The following are equivalent:
\begin{enumerate}
    \item $Q$ is a VC class.
    \item $\lim_{n \to \infty} \frac{g_Q(n)}{n} = 0$
    \item $\lim_{n \to \infty} \frac{r_Q(n)}{n} = 0$
\end{enumerate}
As a consequence, if $Q$ is a VC class, and $\mu$ a probability measure on $X$, then $\lim_{n \to \infty} r_{Q,\mu}(n) = \lim_{n \to \infty} g_{Q,\mu}(n) = 0$. 
\end{lem}
\begin{proof}
The equivalence between (i) and (ii) is given by \cite[Theorem 2.11]{randomVC}, and the equivalence between (ii) and (iii) is evident from Fact \ref{rad_gauss}.
\end{proof}

\begin{defn}
For a function $q \in [0,1]^X$ and $(x_1,\dots,x_n) \in X^n$, define the average $\mathrm{Av}(x_1,\dots,x_n;q) = \frac{1}{n}\sum_{i = 1}^n q(x_i).$

For a function class $Q \subseteq [0,1]^X$, a probability measure $\mu$ on $X$, and $\varepsilon > 0$, say that a tuple $(x_1,\dots,x_n)$ is a $\varepsilon$-\emph{approximation} for $Q$ with respect to $\mu$ when for every $q \in Q$, $|\mathrm{Av}(x_1,\dots,x_n;q) - \mathbb{E}_\mu[q]| \leq \varepsilon$.
\end{defn}

\begin{fact}[{\cite[Theorem 4.10]{wainwright}}]
    Let $Q$ be a class of functions from $X$ to $[0,1]$. Then for any finitely-supported probability measure $\mu$ on $X$, and any $\delta > 0$, we have
    $$\mu^n\left(\sup_{q \in Q}|\mathrm{Av}(x_1,\dots,x_n;q) - \mathbb{E}_\mu[q]|> 2r_{Q,\mu}(n) + \delta\right) \leq \exp\left(-\frac{n\delta^2}{2}\right).$$
\end{fact}

We can use this fact as a version of the VC theorem for function classes:
\begin{thm}\label{thm_approx}
If $n \in \N$ is such that $n > 0$ and $\frac{r_Q(n)}{n}<\varepsilon$, then for any finitely-supported probability measure $\mu$ on $X$, then there exists an $\varepsilon$-approximation for $Q$ in the support of $\mu$ of size at most $n$.

In particular, if $Q$ is a VC class and $\mu$ a finitely-supported probability measure, then for every $\varepsilon > 0$, there exists a $\varepsilon$-approximation for $Q$ in the support of $\mu$, of size at most $n$, where $n = n(\varepsilon,r_Q)$. 
\end{thm}
\begin{proof}
Fix $0 < \delta < \varepsilon - \frac{r_Q(n)}{n}$. Then the probability that a randomly selected tuple $(x_1,\dots,x_n)$ is not an $\varepsilon$-approximation is $$\mu^n\left(\sup_{q \in Q}|\mathrm{Av}(x_1,\dots,x_n;q) - \mathbb{E}_{\mu_1}[q]|> \varepsilon\right) \leq \exp\left(-\frac{n\delta^2}{2}\right) < 1.$$

If $Q$ is a VC class, such a $\varepsilon$ can always be selected for large enough $n$, as $\lim_{n \to \infty} \frac{r_Q(n)}{n} = 0$.
\end{proof}

\subsection{Covering Numbers and $\varepsilon$-Approximations}
In this section, we follow the covering number approach of \cite{alon97} to bound the sizes of $\varepsilon$-approximations, in a measure-theoretic generality suitable for Keisler measures, as in \cite[Section 7.5]{nip_guide}.

\begin{defn}
For $\bar x \in X^n$, let $\mathcal{N}(Q,\bar x,\varepsilon)$ be the $l_\infty$-distance covering number of the set $Q(\bar x)$ -
that is, the minimum size of a set $A \subseteq [0,1]^n$ such that for all $q \in Q(\bar x)$, there is $a \in A$ with $d(a,q) \leq \varepsilon$,
with $d$ denoting the $l_\infty$ distance.

Let $\mathcal{N}_{Q,\varepsilon}(n) = \sup_{\bar x \in X^n} \mathcal{N}(Q,\bar x,\varepsilon)$.
\end{defn}

To bound the covering number, we will use variations on the VC-dimension:
\begin{defn}
Let $Q \subseteq [0,1]^X$ be a class of functions, and $\varepsilon > 0$.

Let the $\varepsilon$-VC-dimension of $Q$, $\mathrm{vc}_\varepsilon(Q)$, be the supremum of the VC-dimensions $\mathrm{vc}(Q_{r,r + \varepsilon})$ where $r \in [0,1-\varepsilon]$.

Define the \emph{fat-shattering} dimension of $Q$, denoted $\mathrm{fs}_\varepsilon(Q)$, to be the maximal cardinality (or $\infty$ if there is no maximum) of a finite set $A \subseteq X$ such that there is a function $f : A \to [0,1]$ such that $(Q - f)_{-\varepsilon,\varepsilon}$ shatters $A$.
\end{defn}

Ben Yaacov \cite{randomVC} has shown that $Q$ is a VC-class if and only if $\mathrm{vc}_\varepsilon(Q)$ is finite for all $\varepsilon > 0$.
The fat-shattering dimension also corresponds (up to constants) to the idea of ``determining a $d$-dimensional $\varepsilon$-box'' in \cite{randomVC}, where it is also shown that $Q$ is a VC-class if and only if $\mathrm{fs}_\varepsilon(Q)$ is finite for all $\varepsilon > 0$.
The following fact relates the two dimensions more concretely:
\begin{fact}[{\cite[Lemma 2.2]{alon97}}]\label{fact_fat}
Let $Q \subseteq [0,1]^X, \varepsilon > 0$. Then
$$\mathrm{vc}_{2\varepsilon}(Q) \leq \mathrm{fs}_{\varepsilon}(Q) \leq \left(2\lceil\frac{1}{\varepsilon}\rceil - 1\right)\mathrm{vc}_\varepsilon(Q).$$
\end{fact}

The fat-shattering dimension is useful for the following lemma.
(The version given here is stated in the proof of the cited lemma.)
\begin{fact}[{\cite[Lemma 3.5]{alon97}}]\label{fact_covering}
Let $\mathrm{fs}_{\varepsilon/4}(Q) \leq d$. Then
$$\mathcal{N}_{Q,\varepsilon}(n) \leq 2\left(\frac{4n}{\varepsilon^2}\right)^{d\log(2en/d\varepsilon)} = n^{O_{d,\varepsilon}(\log n)}.$$
\end{fact}
We can deduce from this and Fact \ref{fact_fat} that the bound of $\mathbb{N}_{Q,\varepsilon}(n) = n^{O_{d,\varepsilon}(\log n)}$ also holds when $\mathrm{vc}_{\varepsilon/4}(Q) \leq d$, although with a different constant.

We can also bound the VC-dimension from the covering numbers.
\begin{lem}\label{lem_vc_covering}
    Let $Q \subseteq [0,1]^X, 0 \leq r < s \leq 1, 0 < \varepsilon < \frac{s - r}{2}$.
    Then $$\pi_{Q_{r,s}}(n) \leq \mathcal{N}_{Q,\varepsilon}(n).$$
\end{lem}
\begin{proof}
Let $A \subseteq X$ be such that $|A| = n$ and $|Q_{r,s} \cap A|$ is maximized, so $|Q_{r,s} \cap A| = \pi_{Q_{r,s}}(n)$.
Then for each subset $B \subseteq A$ in $|Q_{r,s} \cap A|$, there is some $q_B \in Q$ with $q_B(a) \leq r$ for $a \in B$ and $q_B(a)\geq s$ for $a \in A \setminus B$.
The points $(q_B(a) : a \in A)$ for $B \in Q_{r,s} \cap A$ thus each have $\ell_\infty$-distance at least $s - r$ from each other.
Thus no two of them can lie in the same $\varepsilon$-ball in that metric, and the covering number must be at least $\pi_{Q_{r,s}}(n)$.
\end{proof}
In particular, any sub-exponential bound on the covering number for each $\varepsilon$ implies that $Q$ is a VC class of functions.

Alon et al. use the covering number bound to prove the existence of $\varepsilon$-approximations using the following fact:
\begin{fact}[{\cite[Lemma 3.4]{alon97}}]\label{alon34}
Let $\varepsilon > 0, n \geq \frac{2}{\varepsilon^2}, Q \subseteq [0,1]^X$, and let $\bar x = (x_1,\dots,x_n)$ be a tuple of i.i.d. random variables with values in $X$. Then subject to measurability constraints which are satisfied if the probability distribution of each $x_i$ is finitely supported,
$$\mathbb{P}\left[\sup_{q \in Q}\left(\mathrm{Av}(\bar x,q) - \mathbb{E}[q(x_1)]\right) > \varepsilon\right] \leq 12n\mathcal{N}_{Q,\varepsilon/6}(2n)\exp{-\frac{\varepsilon^2 n}{36}}.$$
\end{fact}

Combining the previous two facts gives a bound on the minimum size of an $\varepsilon$-approximation for $Q$ with respect to any finitely-supported probability measure $\mu$:
\begin{fact}[{\cite[Theorem 3.6]{alon97}}]\label{alon36}
Let $Q \subseteq [0,1]^X$ satisfy $\mathrm{fs}_{\varepsilon/24}(Q) \leq d$. Then if $\mu$ is a finitely-supported probability measure on $X$, for all $\varepsilon,\delta > 0$, if $\bar x = (x_1,\dots,x_n)$ consists of i.i.d. random variables with distribution given by $\mu$, we have
$$\mathbb{P}\left[\sup_{q \in Q}\left(\mathrm{Av}(\bar x,q) - \mathbb{E}[q(x_1)]\right) > \varepsilon\right] \leq \delta$$
for
$$n = O\left(\frac{1}{\varepsilon^2}\left(d\ln^2\frac{d}{\varepsilon} + \ln\frac{1}{\delta}\right)\right).$$
\end{fact}

In a forthcoming paper\cite{anderson2}, we will derive version of Facts \ref{alon34} and \ref{alon36} for generically stable Keisler measures in continuous logic,
bounding the sizes of $\varepsilon$-approximations for definable predicates with respect to a fixed generically stable Keisler measure.

\subsection{Transversals and $\varepsilon$-nets}\label{subsection_enets}
While $\varepsilon$-approximations lend themselves naturally to real-valued function classes,
there is another way of approximating set systems with respect to measures that more naturally generalizes to fuzzy set systems: $\varepsilon$-nets.
In this subsection, we will use a fuzzy set system generalization of $\varepsilon$-nets to prove
fuzzy versions of a bound on transversal numbers and to prove a fractional Helly property and $(p,q)$-theorem for fuzzy set systems.
This generalizes the classical combinatorial results for set systems described in \cite[Chapter 10]{matousek_GTM}.

\begin{defn}
    Let $\mathcal{F}$ be a fuzzy set system on $X$, $\mu$ a probability measure on $X$ and $\varepsilon > 0$.
    An $\varepsilon$-net for $\mathcal{F}$ with respect to $\mu$ is a subset $A \subseteq X$ such that for every $(S_+, S_-) \in \mathcal{F}$ such that $\mu(S_+) \geq \varepsilon$, $A \not\subseteq S_-$.
\end{defn}

In order to construct $\varepsilon$-nets out of $\varepsilon$-approximations, we will need to define a construction that crops a function class down to a particular interval. Let $f_{r,s} : [0,1] \to [0,1]$ be the piecewise linear function given by
$$f_{r,s}(x)= \begin{cases}
r & x \leq r \\
x & r \leq x \leq s \\
s & x \geq s
\end{cases}.$$
Now let $Q^{r,s} = \{f_{r,s} \circ q: q \in Q\}$. If $Q$ is a VC-class, then $Q^{r,s}$ will be one as well, and in fact, for any $r'<s'$, the VC-dimension of $(Q^{r,s})_{r',s'}$ will be at most the VC-dimension of $Q_{r',s'}$, and for all $n$, $g_{Q^{r,s}}(n) \leq g_Q(n)$.

\begin{lem}\label{lem_Qnet}
For any $\varepsilon > 0$, $0 \leq r < s \leq 1$, if $Q \subset [0,1]^X$ is a class of functions, $\mu$ is a probability measure on $X$, and $\bar A = (a_1,\dots,a_n)$ is a $\delta$-approximation for $Q^{r,s}$, where $\delta < (s-r)\varepsilon$, then $A = \{a_1,\dots,a_n\}$ is a $\varepsilon$-net for $Q_{r,s}$ with respect to $\mu$.
\end{lem}
\begin{proof}
Fix $\varepsilon, g, Q,$ and $\mu$, let $\delta < (s-r) \varepsilon$, and let $\bar A$ be a $\delta$-approximation for $Q^{r,s}$. Now let $q \in Q$, and assume that $\mu(q_{\leq r}) = \mu((f_{r,s} \circ q)_{\leq r})\geq \varepsilon$. Then $\mathbb{E}_\mu[f_{r,s} \circ q] \leq s-(s - r)\varepsilon$, and accordingly $\mathrm{Av}(a_1,\dots,a_n;q) \leq s-(s - r)\varepsilon + \delta < s$, so there exists at least one $a_i$ with $q(a_i) < s$.
\end{proof}

\begin{thm}\label{thm_Qnet}
For any $\varepsilon > 0$, $0 \leq r < s \leq 1$ and $g : \N \to [0,\infty)$ such that $g(n) = o(1)$, there is $N = N((s-r)\varepsilon, g)$ such that if $Q \subset [0,1]^X$ is a class of functions such that $\frac{r_Q(n)}{n} \leq g(n)$ for all $n$, and $\mu$ is a finitely-supported probability measure on $X$, there is an $\varepsilon$-net $A$ for $Q_{r,s}$ with respect to $\mu$ with $|A| \leq N$.
\end{thm}
\begin{proof}
Let $N,\delta$ be such that $g(N) < \delta < (s-r) \varepsilon$. Using Theorem \ref{thm_approx}, we can find a $\delta$-approximation $\bar A$ for $Q^{r,s}$, which by Lemma \ref{lem_Qnet} is a $\varepsilon$-net for $Q_{r,s}$.
\end{proof}

The bound on the size of $\varepsilon$-nets in Theorem \ref{thm_Qnet} was easy to deduce from a version of the VC-Theorem (Theorem \ref{thm_approx}), but it only applies to fuzzy set systems derived from classes of functions.
With a direct probabilistic argument, adapted from the classical proof by Haussler and Welzl (\cite[Theorem 10.2.4]{matousek_GTM}), we can bound the size on $\varepsilon$-nets for any VC fuzzy set system based only on $\varepsilon$ and the VC-dimension, up to some measurability assumptions.
In an upcoming paper\cite{anderson2}, we will prove that this also holds in the context of generically stable Keisler measures.

\begin{thm}\label{thm_net}
For any $\varepsilon > 0$ and $d \in \N$, there is $N = O(d\varepsilon^{-1}\log \varepsilon^{-1})$ such that if $\mathcal{F}$ is a fuzzy set system on $X$ with VC-dimension at most $d$, and $\mu$ is a finitely-supported probability measure on $X$, there is an $\varepsilon$-net $A$ for $\mathcal{F}$ with respect to $\mu$ with $|A| \leq N$.

If $\mu$ is not necessarily finitely-supported, then the result still holds, assuming the following events are measurable:
$$S_\pm : \, S \in \mathcal{F}$$
$$E_0(x_1,\dots,x_N) = \bigcup_{S \in \mathcal{F}, \mu(S_+)\geq \varepsilon} \left(\bigcap_{i = 1}^N [x_i \in S_-]\right)$$
and
$$E_1(x_1,\dots,x_N,y_1,\dots,y_N) =$$
$$\bigcup_{S \in \mathcal{F}, \mu(S_+)\geq \varepsilon} \left(\bigcap_{i = 1}^N [x_i \in S_-]\right)
    \cap \left(\bigcup_{I \subseteq \{1,\dots,N\},|I|\geq \left\lceil \frac{N\varepsilon}{2}\right\rfloor}\bigcap_{i \in I}[y_i \in S_+]\right).$$

These will be measurable, for instance, if we assume that $\mathcal{F}$ is a countable set system of measurable fuzzy sets, or that $\mu$ is a Borel probability measure on a topological space $X$ where for each $S \in \mathcal{F}$, $S_+$ and $S_-$ are both open.
\end{thm}
\begin{proof}
    This proof generalizes the argument by Haussler and Welzl used in \cite[Theorem 10.2.4]{matousek_GTM}.

Let $N = Cd\varepsilon \log\left(\varepsilon^{-1}\right)$, with $C$ to be determined later. Let $\bar A = (a_1,\dots,a_N)$ be a tuple of independently selected variables with values in $X$ and distribution $\mu$.
Then let $E_0$ be the event that $\{a_1,\dots,a_N\}$ is not a $\varepsilon$-net.
We wish to show that for large enough $C$, $\mathbb{P}[E_0]<1$, so there must exist a $\varepsilon$-net of size $N$.
We can express $E_0 = \bigcup_{S \in \mathcal{F}, \mu(S_+)\geq \varepsilon} \bigcap_{i = 1}^N [a_i \in S_-]$.
If either $\mathcal{F}$ or the support of $\mu$ is countable, then this is clearly measurable, and if each $S_-$ is open, then this is open.

Let $\bar B = (b_1,\dots, b_N)$ be a second tuple of random variables, independent of $\bar A$ with the same distribution.
Let $E_1$ be the event that there exists $S \in \mathcal{F}$ such that $\mu(S_+)\geq \varepsilon$, for each $1 \leq i \leq N$, $a_i \in S_-$ \emph{and} there are at least $k = \lceil \frac{N\varepsilon}{2}\rfloor$ values of $i$ such that $b_i \in S_+$.
We will show that $\mathbb{P}[E_1] \geq \frac{1}{2}\mathbb{P}[E_0]$, and then we will show that $\mathbb{P}[E_1] < \frac{1}{2}$.
We can see that $E_1$ is measurable for the same reasons as $E_0$ is.

To show that $\mathbb{P}[E_1] \geq \frac{1}{2}\mathbb{P}[E_0]$, we will fix $\bar A$, select $B$ conditioned on $\bar A$, and show that $\mathbb{P}[E_1 | \bar A] \geq \frac{1}{2}\mathbb{P}[E_0 | \bar A]$.
If $\{a_1,\dots,a_N\}$ is not a $\varepsilon$-net, then $\mathbb{P}[E_0 | \bar A] = 0$, and as $E_1 \subseteq E_0$, $\mathbb{P}[E_1 | \bar A] = 0$.
Assume $\{a_1,\dots,a_N\}$ is an $\varepsilon$-net.
Then if $I_i$ for $1 \leq i \leq N$ are the indicator random variables for $b_i \in S_+$, and $I = I_1 + \dots + I_N$, we have that $\mathbb{P}[E_1 | \bar A] = \mathbb{P}[I \geq k]$.
The $I_i$s are i.i.d. random variables, equalling 1 with probability $\mu(S_+) \geq \varepsilon$.
By a standard Chernoff tail bound for binomial distributions, we have that $\mathbb{P}[X \geq k] \geq \frac{1}{2} = \frac{1}{2}\mathbb{P}[E_0 | \bar A]$.
Thus in general, $\mathbb{P}[E_1] \geq \frac{1}{2}\mathbb{P}[E_0]$.

To show that $\mathbb{P}[E_1] < \frac{1}{2}$, we will instead condition on the multiset $D = \{a_1,\dots, a_N, b_1,\dots, b_N\}$.
Select $\bar A$ and $\bar B$ by permuting $D$ uniformly at random. All events will be measurable as this probability space is finite.
For any fixed fuzzy set $S \sqsubseteq X$, let $E_S$ be the conditional event that $\bar A \subseteq S_-$ and there are at least $k$ values of $i$ such that $b_i \in S_+$, given the choice of multiset $D$.
We find that if $S'$ is a strong disambiguation of $S \cap D$, then $E_{S} \subseteq E_{S'}$, so if $\mathcal{F}'$ is a strong disambiguation of $\mathcal{F}$ restricted to $D$, we have that
$$E_1|D = \bigcup_{S \in \mathcal{F} : \mu(S_+) \geq \varepsilon} E_S \subseteq \bigcup_{S \in \mathcal{F}} E_S \subseteq \bigcup_{S' \in \mathcal{F}'} E_{S'}.$$
Now we apply Lemma \ref{lem_disamb}, and find a strong disambiguation $\mathcal{F}'$ with $|\mathcal{F}'| = (2N)^{O(d \log (2N))}$, or as we will prefer later, there is $C'$ such that $|\mathcal{F}'| \leq (2N)^{C'(d \log (2N))}$. We find that for each $S' \in \mathcal{F}'$, if $|D \cap S'|< k$, then $\mathbb{P}[E_{S'}] = 0$, and that if $|D \cap S'| \geq k$, then $\mathbb{P}[E_{S'}]$ is the probability that when a set of $N$ elements of $D$ is selected at random, the set is disjoint with $S'$. This is at most
\begin{align*}
    \frac{{2N - |D \cap S'|\choose N}}{{2N \choose N}} 
    \leq \frac{{2N - k\choose N}}{{2N \choose N}} 
    \leq \left(1 - \frac{k}{2N}\right)^N
    \leq e^{- (k/2N)N} = \varepsilon^{Cd/4}.
\end{align*}
Now we bound the probability of the union, letting $C'$ be the constant of :
\begin{align*}
    \mathbb{P}[E_1|D] &\leq \sum_{S' \in \mathcal{F}'} \mathbb{P}[E_{S'}] \\
    &\leq |\mathbb{F}'|\varepsilon^{-Cd/4} \\
    &\leq (2N)^{C'(d \log (2N))}\varepsilon^{Cd/4} \\
    &= \left((2Cd\varepsilon^{-1} \log \varepsilon^{-1})^{C'(\log (2Cd\varepsilon^{-1} \log \varepsilon^{-1}))}\varepsilon^{C/4}\right)^D
\end{align*}
While this expression is somewhat complicated, it is still clear that an increasing quasipolynomial function of $C$ times a decreasing exponential of $C$ will limit to $0$, so for large enough $C$, we find that $\mathbb{P}[E_1|D] < \frac{1}{2}$.
\end{proof}

We apply this first to transversal numbers. We will only apply these to actual discrete set systems, so the definitions here are the same as in \cite{matousek_GTM}.
\begin{defn}
    Let $\mathcal{F}$ be a set system on a set $X$.
    A \emph{transversal} of $\mathcal{F}$ is a set $T \subseteq X$ such that for all $S \in \mathcal{F}$, $T \cap S \neq \emptyset$.
    The \emph{transversal number} of $\mathcal{F}$, $\tau(\mathcal{F})$ is the minimum size of a finite transversal $T \subseteq X$, if it exists.

    A \emph{fractional transversal} of $\mathcal{F}$ is a finitely-supported function $t : X \to [0,1]$ such that for all $S \in \mathcal{F}$, $\sum_{s \in S} t(s) \geq 1$.
    The \emph{fractional transversal number} of $\mathcal{F}$, $\tau^*(\mathcal{F})$ is the minimum size of a fractional transversal $t$, if it exists,
    with the size of $t$ being defined as $\sum_{x \in X} t(x)$.
\end{defn}

We can now use Theorem \ref{thm_net} on the existence of $\varepsilon$-nets of fuzzy set systems to bound the transversal number of the outer set system in terms of the fractional transversal nmuber of the inner set system.
\begin{thm}\label{thm_trans}
Let $d \in \N$, and let $t > 0$. There is $T = T(t,d)$ such that if $\mathcal{F}$ is a finite fuzzy set system on $X$ with VC-dimension at most $d$, and $\tau^*(\mathcal{F}_i) \leq t$, then $\tau(\mathcal{F}_o) \leq T$.
\end{thm}
\begin{proof}
As $\mathcal{F}$ is finite, we may assume that there is an optimal fractional transversal $f : X \to [0,1]$ for $\mathcal{F}_i$ of finite support.
This $f$ leads to a probability measure $\mu$ on $X$ defined by $\mu(\{x\}) = \frac{f(x)}{\tau^*(\mathcal{F}_i)}$ for all $x \in X$, which itself has finite support.

Now we claim that any $\frac{1}{t}-$net for the fuzzy set system $\mathcal{F}$ is a transversal. If indeed a set $A \subseteq X$ is a $\frac{1}{t}-$net, then for any $S \in \mathcal{F}$ such that $\mu(S_+) \geq \frac{1}{\tau^*(\mathcal{F}_i)}$, we also have $\mu(S_+) \geq \frac{1}{t}$, and thus $A \not\subseteq S_-$ by the $\frac{1}{t}-$net property. As for every $S \in \mathcal{F}$, we have $\mu(S_+) = \frac{\sum_{x \in S_+}f(x)}{\tau^*(\mathcal{F}_i)}\geq \frac{1}{\tau^*(\mathcal{F}_i)}$ by the assumption that $f$ is a fractional transversal, $A$ must be a transversal for $\mathcal{F}_o$.

Thus we can simply let $T$ be large enough that there must be a $\frac{1}{t}$-net of size at most $T$. By Theorem \ref{thm_net}, we can choose $T$ depending only on $d$ and $t$.
\end{proof}

We now use $\varepsilon$-nets for fuzzy relations to give a bound on a fuzzy fractional Helly number.
This generalizes the results of \cite{matousek_helly}, using the following definition of a fractional Helly number for a fuzzy relation:
\begin{defn}
    We say that a fuzzy relation $R \sqsubseteq X \times Y$ has \emph{fractional Helly number} $k$ when for every $\alpha > 0$, there is a $\beta > 0$ such that if $b_1,\dots, b_n \in Y$ are such that 
    $\bigcap_{i \in I}R_+^{b_i} \neq \emptyset$ for at least $\alpha{n \choose k}$ sets $I \in {[n] \choose k}$, then there is $J \subseteq [n]$ with $|J| \geq \beta n$ 
    such that $\bigcap_{j \in J}(X \setminus R_-^{b_j}) \neq \emptyset$.
\end{defn}

We can bound the fractional Helly number of a fuzzy relation by its dual VC-density, that is, the exponent of growth of the dual shatter function (the shatter function of the fuzzy set system $S_X$ on $Y$).

\begin{thm}\label{thm_helly}[{Generalizing \cite{matousek_helly}}]
Let $R \sqsubseteq X \times Y$ be a fuzzy set system with $\pi_{R_X}(n) = o(n^k)$. Then $R$ has fractional Helly number $k$.
\end{thm}
\begin{proof}
    This proof follows Matousek's probabilistic argument closely, but it is important to keep track of when an element $S$ of the set system is replaced with $S_+$ or $S_-$.

    Let $\alpha > 0$.
    Fix $m$ such that $\pi_{R_X}(m) < \frac{\alpha}{4}{m \choose k}$, and set $\beta = \frac{1}{2m}$.
    If $n \leq 2m^2 = \frac{m}{\beta}$, then for any $b_1,\dots,b_n \in Y$, 
    all that is required to find a set $J \subseteq [n]$ with $|J| \geq \beta n$ 
    such that $\bigcap_{j \in J}(X \setminus R_-^{b_j}) \neq \emptyset$ is a singleton $J = \{b_j\}$ with $R_+^{b_j} \neq \emptyset$.
    Thus it suffices to show that for $n \geq 2m^2 = \frac{m}{\beta}$, if $b_1,\dots,b_n \in Y$
    are such that $\bigcap_{i \in I}R_+^{b_i} \neq \emptyset$ for at least $\alpha{n \choose k}$ sets $I \in {[n] \choose k}$, then there is $J \subseteq [n]$ with $|J| \geq \beta n$ 
    such that $\bigcap_{j \in J}(X \setminus R_-^{b_j}) \neq \emptyset$.
    
    For contradiction, suppose that $b_1,\dots,b_n \in Y$ satisfy these assumptions, but $\bigcap_{j \in J}(X \setminus R_-^{b_j}) = \emptyset$ for each $J$ with $|J| \geq \beta n$.
    We say that a pair $(J,I)$ with $J \in {[n] \choose m}$, $I \in {J \choose k}$ is \emph{good} when there is $a \in X$ with
    $a \in R_+^i$ for each $i \in I$ and $a \in R_-^j$ for each $j \in J \setminus I$.
    For any given $J$, the set of $I$s such that $(J,I)$ is good is exactly $R_X \cap J$, and by definition, $|R_X \cap J| \leq \pi_{R_X}(m)$,
    As by assumption, $\pi_{R_X}(m) < \frac{\alpha}{4}{m \choose k}$, the probability that $(J,I)$ is good with a randomly chosen $I$ is less than $\frac{\alpha}{4}$.
    
    We now contradict this bound and show that the probability that a randomly chosen $(J,I)$ is good is at least $\frac{\alpha}{4}$.
    Start by choosing $I \in {[n] \choose k}$.
    By assumption, the probability that there is $a \in X$ with $a \in R_+^i$ for each $i \in I$ is at least $\alpha$.
    For each such $i$, fix an $a$, and we will show that when we choose $J \setminus I \in {[n] \setminus I \choose m - k}$ at random,
    $a \in R_-^j$ for each $j \in J \setminus I$ with probability at least $\frac{1}{4}$.
    By assumption, $a \not\in R_-^b$ for less than $\beta n$ values of $b \in \{b_1,\dots,b_n\}$, so the probability that $a \in R_-^j$ for some $j$ is at least
    $$\frac{\left(\lceil (1 - \beta) n\rceil \choose m - k\right)}{{n - k \choose m - k}}
    \geq \prod_{i = 0}^{m - k - 1}\frac{(1 - \beta)n - i}{n - i}
    \geq \prod_{i = 0}^{m - 1}\frac{(1 - \beta)n - m}{n - m}
    \geq \left(\frac{(1 - \beta)n - m}{n - m}\right)^m.$$
    Recalling that $m \leq \beta n$ and $\beta = \frac{1}{2m}$, we see that this is
    $$\left(1 - \frac{\beta n}{n - m}\right)^m \geq \left(1 - 2\beta\right)^m = \left(1 - \frac{1}{m}\right)^m \geq \frac{1}{4}.$$
\end{proof}

We now recall the $(p,q)$ property, a property of classical set systems.
We will use VC-dimension of fuzzy set systems to prove a $(p,q)$-theorem generalizing that of \cite{alon_kleitman}. 
\begin{defn}
    Let $\mathcal{F}$ be a set system on a set $X$. Then $\mathcal{F}$ has the $(p,q)$ \emph{property} when for any $S_1,\dots,S_p \in \mathcal{F}$,
    there are $i_1,\dots,i_q$ such that $\bigcap_{j = 1}^q S_{i_j} \neq \emptyset$.
\end{defn}
If $p=q$, then the $(p,p)$ property just states that any $p$ elements of a set system have nonempty intersection.
We can now adapt the classical proof of the $(p,q)$-theorem, starting with the bound on the fractional transversal number.
\begin{thm}\label{thm_pq}[{Generalizes \cite{alon_kleitman}}]
Let $p\geq q\geq d + 1$ and $0 \leq r < s \leq 1$. Let $\mathcal{F}$ be a finite fuzzy set system with $\mathrm{vc}^*(\mathcal{F}) \leq d$.
If $\mathcal{F}_i$ has the $(p,q)$-property, then $\tau^*(\mathcal{F}_o) \leq N$, where $N = N(p,q,d)$.
\end{thm}
\begin{proof}
We first note that $\tau^*(\mathcal{F}_o) = \nu^*(\mathcal{F}_o)$ when $\mathcal{F}$ is finite, so it suffices to bound $\nu^*(\mathcal{F}_o)$. Now let $f : \mathcal{F} \to [0,1]$ be such that $S_- \mapsto f(S)$ is an optimal fractional packing for $\mathcal{F}_o$, which takes rational values, as $\mathcal{F}$ is finite. (See \cite[Chapter 10]{matousek_GTM}.)

Let $D$ be a common denominator so that $m(S) := Df(S)$ is always an integer. We now define a new fuzzy relation by letting $Y$ be the set of pairs $\{(S,i): S \in \mathcal{F}, 1 \leq i \leq m(S)\}$, and defining $R_m \subseteq X \times Y$ by $(R_m)_+ = \{(a,S,i): a \in S_+\}$ and $(R_m)_- = \{(a,S,i): a \in S_-\}$. Then the inner set system $(R_m)^Y_i$ has the $(p',q)$-property, where $p' = p(d-1)+1$. Let $N = |Y_m| = D\nu^*(\mathcal{F}_o)$.

We claim there exists some $a \in X$ such that $a \not \in (R_m^{(S,i)})_-$ for at least $\beta N$ pairs $(S,i)$ for some $\beta$ depending only on $p$ and $d$. By the fractional Helly theorem, as this class also has VC-codensity at most $d$, it suffices to find $\alpha = \alpha(p,d) > 0$ such that if for at least $\alpha{N \choose k}$ sets $I \in {[N] \choose k}$, there is some $a \in R_m^{y_i}$ for each $i \in I$. Every set of $p'$ sets in this collection contains at least one set of $(d+1)$ sets with nonempty intersection, and each such set of $(d+1)$ sets is contained in ${N - d + 1 \choose p - d + 1}$ sets of $p$ sets. Thus the number of intersecting sets of $(d+1)$ sets from this collection is at least
$$\frac{{N \choose p}}{{N - d + 1 \choose p - d + 1}} \geq \alpha {N \choose d + 1}$$
for some $\alpha = \alpha(p,d) > 0$.

Now since we have $a \in X$ such that $R_m^{(S,i)}$ is not false for at least $\beta N$ pairs $(S,i)$, we have that
$$1 \geq \sum_{S \in \mathcal{F}; a \in S_-} f(S) \geq \sum_{S \in \mathcal{F}; a \in S_-} \frac{m(S)}{D} \geq \frac{1}{D}\beta N = \beta\nu^*(\mathcal{F}_o)$$
so $\nu^*(\mathcal{F}_o) \leq \frac{1}{\beta}$.
\end{proof}

This $(p,q)$ theorem can now be combined with the earlier bound relating the transversal and fractional transversal numbers (Theorem \ref{thm_trans}).
In this process, we end up looking at three nested set systems, using the properties of the innermost to bound the fractional transversal number of the middle set system, and then using that to bound the transversal number of the outermost set system.
To simplify this presentation, we will only give this corollary in the case where the three nested set systems come from the same set of functions, which is exactly the setup we will need for model-theoretic applications:
\begin{cor}\label{cor_pq}
For all $0 \leq r < t < s \leq 1$, $d_1, d_2 \in \N$, and $p \geq q \geq d_1 + 1$, there exists $N = N(d_1,d_2,p,q) \in \N$ such that if $Q \subseteq [0,1]^X$ is a finite function class such that $\mathrm{vc}^*(Q_{r,t}) \leq d_1$ and $\mathrm{vc}(Q_{t,s})\leq d_2$, then for all finite $Q$, if the set system $Q_{\leq r}$ has the $(p,q)$-property, then $\tau(Q_{<s}) \leq N$.
\end{cor}
\begin{proof}
We will first apply Theorem \ref{thm_pq} to the set system $Q_{\leq r}$ to bound $\tau^*(Q_{< t})$, then enlarge the sets slightly without increasing the fractional transversal number, bounding $\tau^*(Q_{\leq t})$, and finally apply Theorem \ref{thm_trans} to bound $\tau(Q_{< s})$.

Fix $p \geq q \geq d_1 + 1$. We will also have $q \geq \mathrm{vc}^*(Q_{r,t}) + 1$. Applying Theorem \ref{thm_pq} now gives us an $N_0$ not depending on $Q$ such that $\tau^*(Q_{< t}) \leq N_0$. As adding to the sets in this set system cannot increase the fractional transversal number, we find that $\tau^*(Q_{\leq t}) \leq \tau^*(Q_{< t}) \leq N_0$.

We now look at the fuzzy set system $Q_{t,s}$. Thus we know that $\tau^*(Q_{\leq t}) \leq N_0$, and it suffices to find $N$ such that $\tau(Q_{< s}) \leq N$. As $\mathrm{vc}(Q_{t,s}) \leq d_2$, Theorem \ref{thm_trans} gives us an $N = N(N_0,d_2)$ such that $\tau((Q^{t,s}_0)_{< 1}) \leq N$.
\end{proof}

\section{Model-Theoretic Preliminaries and Notation}\label{sec_prereq}
We refer to \cite{mtfms} for an introduction to metric structures and continuous logic, although we will need a few additional pieces of notation and background, provided in this section.
Throughout this paper, let $T$ be a theory in continuous logic, using the language $\mathcal{L}$.
We fix a monster model $\mathcal{U} \vDash T$, and will use $M$ to denote a submodel of $\mathcal{U}$,
small in the sense that $\mathcal{U}$ is $|M|^+$-saturated.

In continuous logic, it is natural to deal with variable tuples of countably infinite length. As if $x,y$ are infinite tuples, $|x|$ equals $|x,y|$,
we shall just refer to the relevant cartesian products of a set $M$ as $M^x$ and $M^{x}\times M^{y}$, rather than $M^{|x|}$ or $M^{|x,y|}$.

In classical model theory, we frequently use the notation $\phi(M;b)$ to indicate the subset of $M^x$ defined by the formula $\phi(x;y)$ using the parameter $b \in M^y$.
As this paper will deal with metric structures, where the definable predicate $\phi(x;y)$ can take on any value in $[0,1]$, $\phi(M;b)$ will be defined as the subset of $M^x$ on which $\phi(x;b) = 0$.
For other $r \in [0,1]$, we will use the notations $\phi_{\leq r}(M;b)$ and $\phi_{\geq r}(M;b)$ to denote the sets where $\phi(x;b)\leq r$ and $\phi(x;b)\geq r$.
Given any condition (an inequality or equality of definable predicates), we will use notation such as $[\phi(x) \geq r]$ to denote the subset of a type space $S_x(A)$ where that condition is true.

\subsection{Pairs}
In classical model theory, we frequently add a predicate to pick out a specific subset of a structure, thus making that set definable in the expansion.
In continuous logic, a closed subset of a metric structure is considered definable when its distance predicate is definable. \cite[Def 9.16]{mtfms}
These definable sets are exactly the sets that can be quantified over when constructing definable predicates.
Thus to pick out a particular subset, we restrict our attention to closed subsets, and add a predicate for the distance to that closed subset.

\begin{defn}
    If $M$ is a metric $\mathcal{L}$-structure, and $A \subseteq M^x$ is closed, then let $(M, A)$ be the expansion of $M$ to the language $\mathcal{L}_P$, adding a relation symbol $P$ interpreted as $P(x) = \mathrm{dist}(x,A)$.
\end{defn}
This is a valid metric structure, because $\mathrm{dist}(x,A)$ is bounded and 1-Lipschitz.

Per \cite[Theorem 9.12]{mtfms}, there are axioms indicating that a predicate is the distance predicate of a closed set, so any structure $(N,B)$ elementary equivalent to $(M,A)$ will be an expansion of some $N$ elementarily equivalent to $M$ by a distance predicate for a closed set $B \subseteq N^x$.
Sometimes if $y = (x_1,\dots,x_n)$ or $y = (x_1,x_2,\dots)$, we will use $P(y)$ to denote a definable predicate indicating that $x_i \in A$ for each $i$.
If $y = (x_1,\dots,x_n)$, this can straightforwardly be $P(y) = \max_{i = 1}^n P(x_i)$, but if $y = (x_1,x_2,\dots)$, we may use $P(y) = \sum_{i \in \N} 2^{-i}P(x_i)$,
and we will still have $P(\bar a) = 0$ if and only if $P(a_i) = 0$ for all $i$.

If we wish to define two definable subsets, we will say that $(M,A,B)$ is the expansion adding a distance predicate $P$ to $A$ and a distance predicate $Q$ to $B$.

\subsection{Coding Tricks}
\begin{lem}\label{lem_coding_tricks}
Let $\phi_1(x;y),\dots,\phi_n(x;y)$ be a series of definable predicates, and $A \subseteq \mathcal{U}^y$ be such that $|A|\geq 2$.
Then there is a single definable predicate $\phi(x;y_1,y_2,\dots,y_k)$ such that for every $1 \leq i \leq n$ and $a \in A$,
there is some $\bar a \in A^k$ such that $\phi_i(x;a) = \phi(x;\bar a)$ for all $x$.
\end{lem}
\begin{proof}
Let $a_1,a_2 \in A$ be distinct. Then let $k = 2n+1$ and let 
$$\phi(x;y_1,\dots,y_{2^n}) = \sum_{i = 1}^n \frac{d(y_{2i-1},y_{2i})}{d(a_1,a_2)}\phi_i(x,y_k).$$
Then for any $1\leq i \leq n$ and $a \in A$, we can let $b_k = a$, and choose $b_1,\dots,b_{2n} \in \{a_1,a_2\}$ so that $b_{2j-1} = b_{2j}$ if and only if $j \neq i$.
Then $\phi(x;b_1,\dots,b_k) = \phi_i(x;a)$.
\end{proof}

\subsection{Other Facts}
The following application of the compactness theorem for metric structures will come up in a few proofs later on in the paper.

\begin{lem}\label{lem_dominate}
    Let $A \subseteq \mathcal{U}$. Let $p(x)$ be a partial $A$-type, let $q(x)$ be a partial $\mathcal{U}$-type, and let $\phi(x)$ be a $\mathcal{U}$-definable predicate.
    Then if $p(x) \cup q(x)$ implies $\phi(x) = 0$, there is an $A$-definable predicate $\theta(x)$ such that $p(x)$ implies $\theta(x) = 0$,
    and $q(x)$ implies $\phi(x) \leq \theta(x)$.
\end{lem}
\begin{proof}
    This is a combination of compactness and the proof of \cite[Prop. 7.14]{mtfms}.

    Write $p(x) = \{\psi(x) = 0 : \psi \in \Psi\}$.
    For every $n \in \N$, we see that $\{\psi(x) \leq \delta : \delta > 0, \psi \in \Psi\} \cup q(x) \cup \{\phi(x) \geq 2^{-n}\}$ is inconsistent, so by compactness, there is a subtype $p_n(x) \subseteq p(x)$
    of the form $\{\psi(x) \leq \delta_n : \psi \in \Psi_n\}$ for some $\delta$ and some finite $\Psi_n \subseteq \Psi$ such that $p_n(x) \cup q(x) \cup \{\phi(x) \geq 2^{-n}\}$ is inconsistent.
    Thus if $\theta_n(x) = \max_{\psi \in \Psi_n}\psi(x)$, we see that $p(x)$ implies $\theta_n(x) = 0$, and $\theta_n(x) \leq \delta_n$ implies $\phi(x) < 2^{-n}$.
    Thus also $p(x)$ implies $\sum_{n \in \N}2^{-n}\theta_n(x) = 0$, and for all $n$, $q(x)$ and $\sum_{n \in \N}2^{-n}\theta_n(x) \leq 2^{-n}\delta_n$ implies $\phi(x) < 2^{-n}$.

    Thus by \cite[Prop. 2.10]{mtfms}, there is an increasing continuous function $\alpha : [0,1] \to [0,1]$ such that
    on the subspace of $S_x(\mathcal{U})$ realizing $q(x)$, we have $\phi(x) \leq \alpha\left(\sum_{n \in \N}2^{-n}\theta_n(x)\right)$.
    Thus we may define $\theta(x) = \alpha\left(\sum_{n \in \N}2^{-n}\theta_n(x)\right)$, and we find $q(x)$ implies $\phi(x) \leq \theta(x)$.
\end{proof}
    
\begin{lem}\label{lem_quotient}
    Let $A \subseteq \mathcal{U}$, let $p(x) \subseteq S_x(A)$ be a partial type, and let $\phi(x)$ be a $\mathcal{U}$-definable predicate
    such that for every global type $q(x) \in S_x(\mathcal{U})$ extending $p(x)$, $q|_{A}$ implies $|\phi(x) = r_p|$ for some $r_p$.
    Then there is an $A'$-formula $\psi(x)$ such that $p(x)$ implies $\psi(x) = \phi(x)$.
\end{lem}
\begin{proof}
    The restriction of parameters map $S_x(\mathcal{U}) \to S_x(A)$ is a continuous surjection of compact Hausdorff spaces, and is thus a quotient map.
    The set $[p(x)]$ in either space is closed in $S_x(A)$, and we have assumed that $\phi(x)$, restricted to $[p(x)]\subseteq S_x(\mathcal{U})$,
    lifts to a function from $[p(x)]\subseteq S_x(A)$ to $\R$, which is continuous by the quotient property.
    This continuous function extends to all of $S_x(A)$ by Tietze's extension theorem, and that continuous function is an $A$-definable predicate, $\psi(x)$.
\end{proof}

The following fact about partitions of unity (see \cite[Theorem 2.13]{rudin}) will come up repeatedly in this paper:
\begin{fact}\label{fact_partition_unity}
Let $K$ be a compact Hausdorff space, and let $U_1,\dots, U_n$ be open sets that cover $K$.
Then there are functions $u_1,\dots,u_n : K \to [0,1]$ such that 
\begin{itemize}
        \item for all $x \in K$ and for all $i$, $0 \leq u_i(x) \leq 1$
        \item for all $x \in K$, $u_1(x)+\dots+u_n(x)=1$
        \item for all $i$, the support of $u_i$ is contained in $U_i$ for each $i$.
\end{itemize}
\end{fact}
    
We will also use the notion of a \emph{forced limit} from \cite{local_stab}, in order to carefully define a predicate as a limit of formulas that may not necessarily converge uniformly.
    \begin{defn}
    Let $(a_n : n <\omega)$ be a sequence in $[0,1]$. Define the sequence $(a_{\mathcal{F}\mathrm{lim},n} : n < \omega)$ recursively:
    \begin{align*}
    a_{\mathcal{F}\mathrm{lim},0} &= a_0\\
    a_{\mathcal{F}\mathrm{lim},n+1} &= \begin{cases}
        a_{\mathcal{F}\mathrm{lim},n} + 2^{-n-1}  & \textrm{if\,}\, a_{\mathcal{F}\mathrm{lim},n} + 2^{-n-1} \leq a_{n+1}\\
        a_{n+1} & \textrm{if\,}\, a_{\mathcal{F}\mathrm{lim},n} - 2^{-n-1} \leq a_{n+1} \leq a_{\mathcal{F}\mathrm{lim},n} + 2^{-n-1}\\
        a_{\mathcal{F}\mathrm{lim},n} - 2^{-n-1}  & \textrm{if\,}\, a_{\mathcal{F}\mathrm{lim},n} - 2^{-n-1} \geq a_{n+1}
    \end{cases},
    \end{align*}
    and define the \emph{forced limit} $\mathcal{F}\mathrm{lim}_{n \to \infty} a_n = \lim_{n \to \infty} a_{\mathcal{F}\mathrm{lim},n}$.
    \end{defn}
    
    The authors of \cite{local_stab} make some observations about their construction:
    \begin{fact}[{\cite[Lemma 3.7]{local_stab}}]\label{fact_flim}
        \begin{itemize}
            \item The function $\mathcal{F}\mathrm{lim} : [0,1]^\omega \to [0,1]$ is continuous
            \item If $(a_n : n < \omega)$ is a sequence such that $|a_n - a_{n+1}| \leq 2^{-n}$ for all $n$, then $\mathcal{F}\mathrm{lim} a_n = \lim a_n$
            \item If $a_n \to b$ fast enough that $|a_n - b| \leq 2^{-n}$ for all $n$, then $\mathcal{F}\mathrm{lim} a_n  = b$.
        \end{itemize}
    \end{fact}
    We wish to make one more observation (our technical reason for using this explicit construction):
    \begin{lem}\label{lem_flim}
    If $(a_n : n < \omega)$ is such that $b - 2^{-n} \leq a_n$ for all $n$, then $b \leq \mathcal{F}\mathrm{lim} a_n$.
    \end{lem}
    \begin{proof}
    We just need to show inductively that $b - 2^{-n} \leq a_{\mathcal{F}\mathrm{lim},n}$, as then the limit of this sequence must be at least $b$.
    
    By definition, $b - 2^{-0} \leq a_0 = a_{\mathcal{F}\mathrm{lim},0}$.
    
    Then assume $b - 2^{-n} \leq a_{\mathcal{F}\mathrm{lim},n}$.
    In the three cases of the definition of $a_{\mathcal{F}\mathrm{lim},n+1}$,
    either $a_{\mathcal{F}\mathrm{lim},n+1} \geq a_{n+1}$ or $a_{\mathcal{F}\mathrm{lim},n+1} = a_{\mathcal{F}\mathrm{lim},n} + 2^{-n-1}$.
    In the first case, we have $b - 2^{-n-1} \leq a_{n+1} \leq a_{\mathcal{F}\mathrm{lim},n+1}$,
    and in the second, we have $a_{\mathcal{F}\mathrm{lim},n} + 2^{-n-1} \geq b - 2^{-n} + 2^{-n-1} = b - 2^{-n-1}$.
    \end{proof}
    
    As $\mathcal{F}\mathrm{lim}$ is continuous, it can be used as an infinitary connective on definable predicates.
    That is, if $(\phi_n : n < \omega)$ is a sequence of definable predicates, $\mathcal{F}\mathrm{lim}\phi_n$, defined by pointwise forced limits, will be as well.

\section{NIP and Honest Definitions}\label{sec_HD}
The following definition of NIP for metric structures comes from \cite{randomVC}:
\begin{defn}[IP and NIP]\label{NIP}
    We say a formula $\phi(x;y)$ is \emph{independent} or has \emph{IP} when there exists an indiscernible sequence $(a_i : i \in \omega)$, some tuple $b$, and some $0 \leq r < s \leq 1$ such that for all even $i$, $\vDash \phi(a_i;b) \leq r$ and for all odd $i$, $\vDash \phi(a_i;b) \geq s$.
    
    We say that $T$ is/has \emph{NIP} when no formula $\phi(x;y)$ has IP.
\end{defn}

This indiscernible definition is equivalent to a definition in terms of fuzzy VC-theory, by \cite[Lemma 5.4]{randomVC}.
\begin{fact}
The following are equivalent:
\begin{itemize}
    \item The formula $\phi(x;y)$ is NIP
    \item For all models $\mathcal{M} \vDash T$, the function $\phi(x;y):M^x \times M^y \to [0,1]$ is a VC-function.
\end{itemize}
\end{fact}

We can also give a geometric description of NIP formulas.
If $\phi(x;y)$ is a formula, we can view the set of $\phi$-types $S_\phi(B)$ over some parameter set $B$ as a subset of $[0,1]^B$,
defining it as 
$$S_\phi(B) = \{(\phi(a;b): b \in B): a \in M^x\}.$$
\begin{lem}
    A formula $\phi(x;y)$ has IP if and only if there exists an infinite parameter set $B \subseteq M$ in some $M \vDash T$ such that
    the closed convex hull of $S_\phi(B)$ has nonempty interior in the $\ell_\infty$ metric.
\end{lem}
\begin{proof}
    Suppose that $B$ is infinite and the convex hull of $S_\phi(B)$ has nonempty interior.
    Then for some $\varepsilon > 0$, there is some open $\varepsilon$-ball in the $\ell_\infty$ metric contained in the closed convex hull of $S_\phi(B)$,
    so the closed convex hull of the function class $(\phi(x;b): b \in B)$ on $M^x$ has infinite $\frac{\varepsilon}{2}$-fat-shattering dimension.
    The not-necessarily-closed convex hull will also have infinite $\delta$-fat-shattering dimension for every $\delta< \frac{\varepsilon}{2}$.
    Thus by \cite[Theorem 1.5]{convexFS}, which places a bound on the $\delta$-fat-shattering dimension of a convex hull in terms of the $\frac{\delta}{4}$-fat-shattering dimension of the larger class,
    we see that the $\frac{\varepsilon}{8}$-fat-shattering dimension of $(\phi(x;b): b \in B)$ is infinite, so $\phi(x;y)$ has IP.

    Suppose that $\phi(x;y)$ has IP. Then there is some $\varepsilon > 0$ such that the $\varepsilon$-fat-shattering dimension of $\phi(x;y)$ is infinite in some model $M \vDash T$.
    This means that the partial type on variables $(x_\sigma : \sigma \in \{0,1\}^\N; y_n : n \in \N)$ consisting of $\phi(x_\sigma;y_n) + 2\varepsilon \leq \phi(x_\tau;y_n)$ for each $n \in \N$
    and $\sigma, \tau$ that are equal except for the $n$th coordinate where $\sigma(n) = 0$ and $\tau(n) = 1$ is consistent, so we can find some
    $(a_\sigma : \sigma \in \{0,1\}^\N; b_n : n \in \N)$ realizing this type.
    Then if $B = \{b_n : n \in \N\}$, the convex hull of $(\mathrm{tp}_\phi(a_\sigma;B): \sigma \in \{0,1\}^\N)$ will contain a $\varepsilon$-ball.
\end{proof}

This section is dedicated to defining a continuous logic version of a third equivalent definition of NIP, honest definitions, and proving its equivalence to the others.

\begin{defn}\label{defn_HD}
Let $A$ be a closed subset of $M^x$ where $M \preceq \mathcal{U}$ and $(M,A)\preceq (M',A')$.
Let $\phi(x;b)$ be an $M$-predicate, and let $\psi(x;d)$ be an $A'$-predicate.
We say that $\psi(x;d)$ is an \emph{honest definition} for $\phi(x;b)$ over $A$ when 
\begin{itemize}
    \item for all $a \in A$, $\phi(a;b) = \psi(a;d)$
    \item for all $a' \in A'$, $\phi(a';b) \leq \psi(a';d)$.
\end{itemize}

If the same predicate $\psi(y;z)$ works for any choice of $M,A,b$ with $|A| \geq 2$, then we call $\psi(x;z)$ an honest definition for $\phi(x;y)$.
Also, because we are only concerned with honest definitions with parameters in $A \subseteq M^x$, we assume that $z = (x_1,\dots,x_n)$ or $z = (x_1,x_2,\dots)$.
In either case, we abuse notation slightly and use $A^z$ to refer to $A^n$ or $A^\N$ in those respective cases.

For all $\phi(x;y)$ and $\psi(y;z)$, we also define a predicate 
$$\mathrm{HD}_{\phi,\psi,P,Q}(y;z) = \max\left(\sup_{x : P(x)}|\phi(x;y) - \psi(x;z)|,\sup_{x : Q(x)}\phi(x;y) \dot{-} \psi(x;z)\right).$$
Then for $d \in A'^z$, $(M',A,A') \vDash \mathrm{HD}_{\phi,\psi,P,Q}(b;d)$ if and only if $\psi(x;d)$ is an honest definition for $\phi(x;b)$.
We will abuse notation later to write $\mathrm{HD}_{\phi,\psi,A,A'}(b;d)$ for the value of $\mathrm{HD}_{\phi,\psi,P,Q}(b;d)$ in $(M',A,A')$.
\end{defn}

In classical logic, all predicates are formulas, and only take values 0 and 1 corresponding to true and false.
Then our definition of $\psi(x;d)$ being an honest definition for $\phi(x;b)$ over $A$ corresponds to
$$\phi(A;b) \subseteq \psi(A';d) \subseteq \phi(A';b),$$
which is how honest definitions are presented in \cite[Theorem 3.13]{nip_guide}.

Because the property of $\psi(x;d)$ being an honest definition for $\phi(x;b)$ is encapsulated in $\mathrm{HD}_{\phi,\psi,A,A'}(b;d)$, we see that it does not depend on the choice of $(M',A')$, as long as $d \in A'^z$.
On our way to honest definitions, it will sometimes be easier to work with $\psi(x;d)$ such that $\mathrm{HD}_{\phi,\psi,A,A'}(b;d)$ is small, but not necessarily zero.
In fact, finding such $\psi(x;d)$ with $\mathrm{HD}_{\phi,\psi,A,A'}(b;d)$ arbitrarily small implies the existence of an honest definition.

\begin{lem}\label{lem_HD_lim}
    Let $A$ be a closed subset of $M^x$ where $M \preceq \mathcal{U}$, and let $\phi(x;b)$ be an $M$-predicate.
    Let $(M,A)\preceq (M',A')$, and let $d \in A'^z$.
    Let $(\psi_n(x;z): n \in \N)$ be a sequence of definable predicates with
    $\mathrm{HD}_{\phi,\psi_n,A,A'}(b;d) \leq 2^{-n}$ for each $n$.
    Then $\mathcal{F}\lim\psi_n(x;d)$ is an honest definition for $\phi(x;b)$ over $A$.

    If instead we have a sequence $(\psi_n(x;z_n): n \in \N)$ with different $d_n \in A'^{z_n}$ such that
    $\mathrm{HD}_{\phi,\psi_n,A,A'}(b;d_n) \leq 2^{-n}$ for each $n$,
    then $\mathcal{F}\lim\psi_n(x;d)$ is an honest definition for $\phi(x;b)$ over $A$,
    where $d$ is a concatenation of all the tuples $d_n$.
\end{lem}
\begin{proof}
    Let $\psi(x;z) = \mathcal{F}\lim\psi_n(x;z)$.
If $a \in A$, we have $|\psi_n(a;d) -\phi(a;b)|\leq 2^{-n}$, so by Fact \ref{fact_flim}, $\mathcal{F}\lim\psi_n(a;d) = \phi(a;b)$.
If $a' \in A'$, we have $\phi(a';b)- 2^{-n} \leq \psi_n(a';d)$, so by Lemma \ref{lem_flim}, $\phi(a';b) \leq \mathcal{F}\lim\psi_n(a';d)$.
Thus $\psi(x;d)$ is an honest definition of $\phi(x;b)$ over $A$.

If each $\psi_n(x;z_n)$ uses on different parameters $d_n$, and we let $z$ be a concatenation of all variable tuples $z_n$, with $d$ a concatenation of all the parameters $d_n$,
then for each $n$, we can think of $\psi_n$ as a predicate $\psi_n(x;z)$ with $\psi_n(x;d) = \psi_n(x;d_n)$.
Then we have $|\psi_n(a;d) -\phi(a;b)|\leq 2^{-n}$, so defining $\psi(x;z) = \mathcal{F}\lim\psi_n(x;z)$ we still get that $\psi(x;d)$ is an honest definition of $\phi(x;b)$ over $A$.
\end{proof}

\begin{thm}\label{thm_HD}
    Assume $T$ is NIP. Let $\mathcal{M} \models T$, $A \subseteq M^x$ closed, $\phi(x)$ a definable predicate with parameters in $M$.
    Then $\phi(x)$ admits an honest definition over $A$.
\end{thm}
\begin{proof}
        Let $(M', A')$ be a $|M|^+$-saturated elementary extension.
        
        We use the set $S_A \subseteq S_x(\mathcal{U})$ of types approximately realizable in $A$, and the fact that $S_A$ is compact.
        We will replace this with the set of \emph{approximately realized} types, as in \cite[Def. 3.1]{StabGrps}.
        In Fact 3.3, it is established that the set of all such types is in fact closed, and thus compact.
        Let $p \in S_A$, and let $\phi(p)$ be the unique value of $\phi(a)$ for $a \vDash p$.
        We will first show that $p|_{A'}(x)$ and $\{P(x) = 0\}$ implies $\phi(x) = \phi(p)$.

        Fix $\varepsilon > 0$.
        We will try to build a Morley sequence $(a_i : i \in \omega)$ for $p$ over $A$ in $A'$ that contradicts NIP, by satisfying these properties:
        \begin{itemize}
            \item $P(a_i) = 0$
            \item $a_i \models p|_{Aa_{<i}}$
            \item $\models |\phi(a_{i + 1}) - \phi(a_i)| \geq \frac{\varepsilon}{2}$
        \end{itemize}
        If we can build such a sequence, it will be indiscernible over $A$, and will thus violate NIP.
        Thus for some $i$, the partial type $p|_{Aa_{\leq i}} \cup \{P(x) = 0\} \cup \{|\phi(x) - \phi(a_i)| \geq \frac{\varepsilon}{2}\}$ is not consistent.
        We see that $p$ must not contain the formula $|\phi(x) - \phi(a_i)| \geq \frac{\varepsilon}{2}$, or else this would be a subset of $p \cup \{P(x) = 0\}$, which is consistent as $p$ is approximately realizable in $A$.
        Thus $|\phi(p) - \phi(a_i)|<\frac{\varepsilon}{2}$, and thus the partial type $p|_{Aa_{\leq i}} \cup \{P(x) = 0\} \cup \{|\phi(x) - \phi(p)| \geq \varepsilon\}$ is not consistent.
        As this means $p|_{A'} \cup \{P(x) = 0\}$ implies $|\phi(x) - \phi(p)| < \varepsilon$ for every $\varepsilon > 0$, we see that $p|_{A'} \cup \{P(x) = 0\}$ implies $\phi(x) = \phi(p)$.
        By Lemma \ref{lem_quotient}, there is an $A'$-definable $\mathcal{L}_P$ predicate $\psi_P(x;d_1)$ in the pair language such that $S_A(x)$ and $P(x) = 0$ imply $\phi(x) = \psi_P(x;d_1)$.
        Thus by replacing each instance of the predicate $P(x)$ in $\psi_P(x;z)$ with 0, we get an $A$ definable $\mathcal{L}$-predicate $\psi_0(x;z)$ with $P(x) = 0$ implying $\psi_0(x;z) = \psi_P(x;z)$.
        
        This means that $S_A(x)$ and $P(x) = 0$ imply $\phi(x) = \psi_0(x;d_1)$, so if we let $\theta_0(x) = \phi(x) \dot{-} \psi_0(x;d_1)$, we see that $S_A(x)$ and $P(x) = 0$ imply $\theta_0(x) = 0$,
        so there is some $A'$-definable $\theta(x;d_2)$ with $S_A(x)$ implying $\theta(x;d_2) = 0$ and $P(x)$ implying $\theta(x;d_2) \geq \theta_0(x)$.
        Thus letting $\psi(x;d) = \psi_0(x;d_1) + \theta(x;d_2)$, we see that for $a \in A$, as $S_A(a)$ holds and $P(a) = 0$, $\psi(a;d) = \psi_0(a;d_1) + \theta(a;d_2) = \phi(a)$,
        and for $a \in A'$, as $P(a') = 0$, we have $\psi(a;d) \geq \psi_0(a;d_1) + (\phi(a) \dot{-} \psi_0(a;d_1)) \geq \phi(a)$.
\end{proof}

    In order to uniformize honest definitions, we will work with series of approximations to honest definitions over finite sets.
\begin{lem}\label{lem_FHD}
Let $A$ be a closed subset of $M^x$ where $M \preceq \mathcal{U}$, and let $\phi(x;b)$ be an $M$-predicate.
Fix $(M,A) \preceq (M',A')$ to be $|M|^+$-saturated, $\varepsilon > 0$, and a definable predicate $\psi(x;z)$.

If there exists $d \in A'^z$ such that $\mathrm{HD}_{\phi,\psi,A,A'}(b;d) < \varepsilon$, then for all finite $A_0 \subseteq A$,
we have there is a tuple $d_{A_0} \in A^z$ such that $\mathrm{HD}_{\phi,\psi,A_0,A}(b;d_{A_0}) < \varepsilon$.

Conversely, if for all finite $A_0 \subseteq A$, there is a tuple $d_{A_0} \in A^z$ such that $\mathrm{HD}_{\phi,\psi,A_0,A}(b;d_{A_0}) \leq \varepsilon$,
then there exists $d \in A'^z$ such that $\mathrm{HD}_{\phi,\psi,A,A'}(b;d) \leq \varepsilon$.
\end{lem}
\begin{proof}
    First, we observe that for finite $A_0$, $\mathrm{HD}_{\phi,\psi,A_0,A}(b;z)$ is equivalent to the predicate
    $$\max\left(\max_{a_0 \in A_0}|\phi(a_0;y) - \psi(a_0;z)|,\sup_{x : P(x)}\phi(x;y) \dot{-} \psi(x;z)\right)$$
    which is expressible using only the predicate $P$.
    Thus by elementarity, for any $d \in A^z$, $\mathrm{HD}_{\phi,\psi,A_0,A}(b;d) = \mathrm{HD}_{\phi,\psi,A_0,A'}(b;d)$ and
    $\inf_{z \in A^z}\mathrm{HD}_{\phi,\psi,A_0,A}(b;z) = \inf_{z \in A'^z}\mathrm{HD}_{\phi,\psi,A_0,A'}(b;z)$.

    Now fix $\varepsilon > 0$. Assume that $d \in A'^z$ is such that $\mathrm{HD}_{\phi,\psi,A,A'}(b;d) < \varepsilon$.
    Then because $\inf$ corresponds to $\E$ for open conditions, $\inf_{z \in A'^z}\mathrm{HD}_{\phi,\psi,A,A'}(b;z) < \varepsilon$, and for every finite $A_0 \subseteq A$, $\inf_{z \in A'^z}\mathrm{HD}_{\phi,\psi,A_0,A'}(b;z) < \varepsilon$.
    Then by elementarity, $\inf_{z \in A^z}\mathrm{HD}_{\phi,\psi,A_0,A}(b;z) < \varepsilon$,
    so there is some $d_{A_0} \in A^z$ such that $\mathrm{HD}_{\phi,\psi,A_0,A}(b;d_{A_0}) < \varepsilon$.

    Now, assume that for all finite $A_0 \subseteq A$, there is a tuple $d_{A_0} \in A^z$ such that $\mathrm{HD}_{\phi,\psi,A_0,A}(b;d_{A_0}) \leq \varepsilon$.
    
    Let $(M',A')$ be a $|M|^+$-saturated elementary extension of $(M,A)$.
    We claim that for some fixed $d \in A'^z$, $\mathrm{HD}_{\phi,\psi,A,A'}(b;d) \leq \varepsilon$
    if and only if $\mathrm{HD}_{\phi,\psi,A_0,A'}(b;d) \leq \varepsilon$ for each $A_0 \subseteq A$.
    This is because both inequalities are equivalent to stating that $\phi(a';b) \leq \psi(a';d) + \varepsilon$ for all $a' \in A'$, as well as
    stating that $|\phi(a;b) - \psi(a;d)| \leq \varepsilon$ for all $a \in A$.
    Thus to find $d \in A'^z$ with $\mathrm{HD}_{\phi,\psi,A,A'}(b;d) \leq \varepsilon$,
    it suffices to show that the partial type 
    $$p(z) = \left\{\mathrm{HD}_{\phi,\psi,A_0,A'}(b;z) \leq \varepsilon : A_0 \subseteq A\right\}$$
    is consistent, which it is, as any finite subtype is implied by a single condition $\mathrm{HD}_{\phi,\psi,A_0,A'}(b;z) \leq \varepsilon$ with $A_0$ finite,
    which is equivalent to $\mathrm{HD}_{\phi,\psi,A_0,A}(b;z) \leq \varepsilon$, which is realized by some $d_{A_0} \in A^z$.
\end{proof}

Now we work towards uniformizing Honest Definitions, using the characterization over finite sets, so that we can use the same formula $\psi$ for all sets $A$.

\begin{lem}\label{lem_UHD}
Assume $T$ is NIP.

Let $\phi(x;y)$ be a formula, $\varepsilon > 0$, and let $\psi \mapsto q_\psi$ be a function from the set of definable predicates $\psi(x;z)$ to $\N$.
Then there are finitely many formulas $\psi_0,\dots,\psi_{n-1}$ such that:

For any $\mathcal{M} \models T$, $A \subseteq M$ closed, $b \in M^y$,
there exists $j < n$ such that for any $A_0 \subseteq A$ of size $|A_0| \leq q_{\psi_j}$,
$\inf_{z : P(z)}\mathrm{HD}_{\phi,\psi,A_0,A}(b;z) < \varepsilon$.
\end{lem}
\begin{proof}
We work in the extended language $\mathcal{L} \cup \{P(x),c_b\}$, where $c_b$ is a tuple of constants of the same cardinality as $z$.

For each $\psi(x;z)$, let $\Theta_{\psi}$ be
$$\sup_{x_0,\dots,x_{q_{\psi} - 1} \in P} \left(\inf_{z : P(z)}\left(\max\left(\max_{i<q_\psi}|\phi(x_i;c_b) - \psi(x_i;z)|,\sup_{x : P(z)}\phi(x;c_b) \dot{-}\psi(x;z)\right)\right)\right).$$

This formula is defined so that in an expansion $(M,A,b)$ of a model $M \vDash T$,
$$\Theta_\psi = \sup_{A_0 \subseteq A : |A_0| \leq q_\psi}\inf_{z : P(z)}\mathrm{HD}_{\phi,\psi,A_0,A}(b;z).$$

For each model $(M,A,b)$ of this extended language, by Theorem \ref{thm_HD}, there is an honest definition $\psi(x;d)$ of $\phi(x;b)$ over $A$, so $\mathrm{HD}_{\phi,\psi,A,A'}(b;d) = 0 \leq \frac{\varepsilon}{2}$.
Thus by Lemma \ref{lem_FHD}, for all finite $A_0 \subseteq A$, $\inf_{z \in A^z}\mathrm{HD}_{\phi,\psi,A_0,A}(b;z) < \frac{\varepsilon}{2}$.
Thus the supremum over all such $A_0$ is at most $\frac{\varepsilon}{2}$, and $(M,A,b) \vDash \Theta_\psi \leq \frac{\varepsilon}{2}$.

As at least one of the open conditions $\{\Theta_\psi< \varepsilon : \psi(x;z)\}$ holds in every model, this set covers the (zero-variable) type space.
Thus by compactness, there is a finite collection $\psi_0,\dots,\psi_{n-1}$ such that one of the open conditions $\Theta_{\psi_j}<\varepsilon$ is true in each model $(M,A,b)$.
Unpacking the definition of $\Theta_\psi$, this yields the result.
\end{proof}

We can now apply Corollary \ref{cor_pq} to finish uniformizing Honest Definitions:
\begin{thm}\label{thm_UHD}
Assume $T$ is NIP.
Every definable predicate $\phi(x;y)$ admits an honest definition $\psi(x;z)$.
\end{thm}
\begin{proof}
    For each $\varepsilon > 0$, we will find $\psi(x;z)$ such that for every $A,b$, and any finite $A_0 \subseteq A$, there is some $d \in A^z$ such that
    $\mathrm{HD}_{\phi,\psi,A_0,A}(b;d) < \varepsilon$.
    Then by Lemma \ref{lem_FHD}, for every $A,b$ and saturated extension $(M,A) \preceq (M',A')$, there will be $d \in A'^z$ with $\mathrm{HD}_{\phi,\psi,A_0,A}(b;d) \leq \varepsilon$.
    If for each $n$, we choose $\psi_n(x;z_n)$ that works for $\varepsilon = 2^{-n}$, then by Lemma \ref{lem_HD_lim}, the forced limit $\mathcal{F}\lim \psi_n(x;z)$ will be a uniform honest definition.

    Fix $\varepsilon > 0$.
If there are finitely many predicates $\psi_0,\dots,\psi_{n-1}$ such that for each $(M,A,b,A_0)$, one suffices,
we can use the standard coding tricks (see Lemma \ref{lem_coding_tricks}) to find a single $\psi$ that can code all of these, provided $|A| \geq 2$.

Having made all these reductions, we now find candidate predicates using Lemma \ref{lem_UHD}.
Given a partitioned predicate $\psi(x;z)$, let 
$q_\psi = \mathrm{vc}^*_{\frac{\varepsilon}{2},\frac{3\varepsilon}{4}}(|\phi(x,y) - \psi(x;z)|) + 1,$
where we view $|\phi(x,y) - \psi(x;z)|$ as partitioned between variables $(x,y)$ and $z$.
Now let $\psi_0',\dots,\psi_{n-1}'$ be the predicates given by Lemma \ref{lem_UHD} such that
for any $\mathcal{M} \models T$, $A \subseteq M$ closed, $b \in M^y$,
there exists $j < n$ such that for any $A_0 \subseteq A$ of size $\leq q_{\psi_j'}$, $\inf_{z : P(z)}\mathrm{HD}_{\phi,\psi_j',A_0,A}(b;z) < \frac{\varepsilon}{2}$.

Now fix $M,A,b,A_0$.
We know that for some $j < n$, and for all $A_0' \subseteq A_0$ of size $\leq q_{\psi_j'}$,
there is some $d \in A^z$ such that $\mathrm{HD}_{\phi,\psi_j',A_0',A}(b;d) < \frac{\varepsilon}{2}$.

Let $D = \{d \in A^z : \forall a \in A, \phi(a;b)< \psi_j'(a;d) + \frac{\varepsilon}{2}\}$.
Let $Q$ be the finite function class on $D$ consisting of the functions 
$\{|\phi(a_0;b) - \psi_j'(a_0;z)|: a_0 \in A_0\}$.
Then for any $r,s$, we have $\mathrm{vc}^*(Q_{r,s}) \leq \mathrm{vc}^*_{r,s}(|\phi(x,y) - \psi_j'(x;z)|)$, as $Q$ consists of fewer functions on a restricted domain.
In particular, $q_{\psi_j'} \geq \mathrm{vc}^*(Q_{\frac{\varepsilon}{2},\frac{3\varepsilon}{4}}) + 1$, and
$Q_{\leq \frac{\varepsilon}{2}}$ has the $(q_{\psi_j'}, q_{\psi_j'})$ property,
so by Corollary \ref{cor_pq}, there is some $N$ depending only on
$\mathrm{vc}^*_{\frac{\varepsilon}{2},\frac{3\varepsilon}{4}}(|\phi(x,y) - \psi_j'(x;z)|)$
and
$\mathrm{vc}^*_{\frac{3\varepsilon}{4},\varepsilon}(|\phi(x,y) - \psi_j'(x;z)|)$
such that $\tau(Q_{<\varepsilon}) \leq N$.
That is, there exist $d_1,\dots,d_N \in D$ such that for each $a \in A_0$, there is some $d_i$ with
$|\phi(a;b) - \psi_j'(a;d_i)|<\varepsilon$.

Now we let $\psi_j(x;z_1,\dots,z_N) = \min_{1 \leq i \leq N} \psi'_j(x;z_i)$, remembering that $N$ depends only on $\phi,\psi_j'$.
It suffices to show that $\mathrm{HD}_{\phi,\psi_j,A_0,A}(b;d_1,\dots,d_N) < \varepsilon$.
We see that $\psi_j(x;d_1,\dots,d_N)$ satisfies 
for all $a \in A$, $\phi(a;b)< \psi_j(a;d_1,\dots,d_N) + \varepsilon$, as for each $i$,
$\phi(a;b)< \psi_j'(a;d_i) + \frac{\varepsilon}{2}$, so we have taken a minimum of functions that are all sufficiently large.
Also, for each $a \in A_0$, there exists some $d_i$ with $\psi_j'(a;d_i) <\phi(a;b) + \varepsilon$,
so taking the minimum $\psi_j(a;d_1,\dots,d_N) <\phi(a;b) + \varepsilon$, and
$|\phi(x;b) - \psi_j(a;d_1,\dots,d_N)|<\varepsilon$.
\end{proof}

We now get a version of uniform definability of types over finite sets (UDTFS).
\begin{defn}
    Let $\phi(x;y)$ be a definable predicate.
    Then we say $\phi(x;y)$ has UDTFS when there is a definable predicate $\psi(x;z)$ (where $z$ consists of $k$ copies of $x$, where $k$ is possibly infinite) such that for any finite $A \subseteq \mathcal{U}^x$ with $|A| \geq 2$, and any $b \in \mathcal{U}^y$, there is $d$ in $A^k$ such that $\phi(a;b) = \psi(a;d)$ for all $a \in A$.
\end{defn}

\begin{cor}[UDTFS]\label{cor_UDTFS}
Assume $T$ is NIP. Every definable predicate $\phi(x;y)$ has UDTFS.
\end{cor}
\begin{proof}
    Simply let $\psi(x;y)$ be an honest definition of $\phi(x;z)$.
\end{proof}

UDTFS also provides polynomial bounds on covering numbers.

\begin{lem}\label{lem_nip_covering}
    Let $\phi(x;y)$ be a formula such that $\phi(x;y)$ has UDTFS, with uniform definition $\psi(x;z)$.
    Let $\varepsilon > 0$, and let $\psi_\varepsilon(x;z)$ be a formula depending only on a finite number $k$ of the variables of $z$ such that
    $\vDash \sup_x \sup_z|\psi(x;z) - \psi_\varepsilon(x;z)| \leq \varepsilon$.
    Then $\mathcal{N}_{\phi(x;y),\varepsilon}(n) = O(n^k)$.
    (In fact, $\mathcal{N}_{\phi(x;y),\varepsilon}(n) \leq n^k$ for $n \geq 2$.)
\end{lem}
\begin{proof}
    Recall that $\mathcal{N}_{\phi(x;y),\varepsilon}(n)$ is the supremum of the $\varepsilon$-covering numbers in the $\ell_\infty$-metric of the sets
    $\phi(\bar a;y) = \{(\phi(a_i;b) :1 \leq i \leq n) : b \in \mathcal{U}^y\}$ for $\bar a = (a_1,\dots,a_n) \in (\mathcal{U}^x)^n$.
    
    Fix $\bar a$, and let $A = \{a_1,\dots,a_n\}$.
    If $|A| = 0$, then $n = 0$, and this is trivial.
    If $|A| = 1$, then there is a $\varepsilon$-cover of size at most $\mathcal{N}_{\phi(x;y),\varepsilon}(1)$, a constant.
    
    Now assume $|A| \geq 2$.
    By UDTFS, the set $\phi(\bar a;y)$ equals the set $\psi(\bar a;z)$.
    Let $z_0 \subseteq z$ be the finite tuple with $|z_0| = k$ on which $\psi_\varepsilon$ depends.
    Let $\pi : A^z \to A^{z_0}$ be the restriction map, and let $D \subseteq A^z$ be such that $\pi$ is bijective on $D$.
    Thus $|D| = |A^{z_0}| \leq n^k$.
    Then $\{(\psi_\varepsilon(a_i;d) :1 \leq i \leq n): d \in D\}$ is a $\varepsilon$-cover for $\psi(\bar a;z)$, as for every
    $d \in \mathcal{U}^y$, there is some $d' \in D$ with $\pi(d) = \pi(d')$, and thus $\psi_\varepsilon(d) = \psi_\varepsilon(d')$,
    so in turn, for all $a \in A$, $|\psi(a;d) - \psi_\varepsilon(a;d')| \leq \varepsilon$.
    Thus $(\psi_\varepsilon(a_i;d') :1 \leq i \leq n)$ is within $\varepsilon$ of $(\psi(a_i;d) :1 \leq i \leq n)$ in the $\ell_\infty$-metric.
\end{proof}

We now tie UDTFS back into a characterization of NIP.
\begin{lem}
Let $\phi(x;y)$ be a definable predicate, and assume that $\phi(x;y)$ has UDTFS.
Then $\phi(x;y)$ is NIP.
\end{lem}
\begin{proof}
    By Lemma \ref{lem_vc_covering}, the polynomial bound given by Lemma \ref{lem_nip_covering} on the covering number shows that $\phi(x;y)$ is a VC-class of functions.
\end{proof}

This gives us several equivalent characterizations of NIP:
\begin{thm}\label{thm_NIP_equiv}
The following are equivalent:
\begin{itemize}
    \item $T$ is NIP
    \item every definable predicate $\phi(x;y)$ admits an honest definition $\psi(x;z)$
    \item $T$ has UDTFS.
\end{itemize}
\end{thm}

It remains to be checked whether a given formula or predicate $\phi(x;y)$ being NIP guarantees uniformity of honest definitions and UDTFS, although this was recently established for discrete logic in \cite{udtfs}.

\subsection{The Shelah Expansion}
We now propose definitions of externally definable predicates and the Shelah expansion in continuous logic.
We confirm that it preserves NIP, as in classical logic, using a generalization of the honest definitions proof from \cite{cs1}.

\begin{defn}[External definability]
Let $M$ be a metric $\mathcal{L}$-structure.
We say that a function $f : M^x \to [0,1]$ is an \emph{externally definable predicate} when there is some elementary extension $M \preceq N$, some definable predicate $\phi(x;y)$, and some $b \in N^y$
such that $f(a) = \phi(a;b)$ for all $a \in M^x$.

If $\phi(x;y)$ can be chosen to be a formula rather than just a definable predicate, we say that $f$ is \emph{externally formula-definable}.
\end{defn}

\begin{defn}[The Shelah Expansion]
Let $M$ be a metric $\mathcal{L}$-structure, with $M \preceq N$ a $|M|^+$-saturated elementary extension.
We define $M^{\mathrm{Sh}}$, the \emph{Shelah expansion} of $M$, to be the metric structure consisting of the same underlying metric space $(M,d)$,
together with a predicate symbol $P_{\phi,b}(x)$ for each $\mathcal{L}$-formula $\phi(x;y)$ and $b \in N^y$, interpreted so that
$P_{\phi,b}(a) = \phi(a;b)$ for all $a \in M$.
The formula $P_{\phi,b}$ is assigned a Lipschitz constant $C$ such that $\phi(x;y)$ is provably $C$-Lipschitz.
Denote this language $\mathcal{L}^{\mathrm{Sh}}$.
\end{defn}

\begin{lem}\label{lem_sh_predicate}
Fix $M \preceq N$ with $N$ $|M|^+$-saturated.
Then the predicates $M^x \to [0,1]$ given by quantifier-free formulas $\phi(x)$ in $\mathcal{L}^{\mathrm{Sh}}$
are exactly the externally formula-definable predicates on $M$,
and the quantifier-free $\mathcal{L}^{\mathrm{Sh}}$-definable predicates on $M^{\mathrm{Sh}}$ are precisely the externally definable predicates on $M$.
\end{lem}
\begin{proof}
By definition, any externally (formula-)definable predicate $f : M^x \to [0,1]$ is given by $\phi(x;b)$ for some formula/definable predicate $\phi(x;y)$ and some $b \in N'^y$ where $M \preceq N'$.
For any $b'$ in any extension of $M$, $\phi(x;b')$ defines $f$ if and only if $b$ realizes the partial type $p(y) = \{\phi(a;y) = f(a) : a \in M^x\}$.
This partial type is realized by $b$, so by saturation, it is realized by some $b' \in N$, so $f$ is externally (formula-)definable with parameters in $N$.
Thus the choice of $N$ does not matter, and it suffices to consider parameters in a fixed $N$.

Thus if $f$ is externally formula-definable, we may choose a formula $\phi(x;y)$ and $b \in N^y$ such that $f(x) = \phi(x;b) = P_{\phi,b}(x)$ on $M$, so
$f$ is given by a formula in $\mathcal{L}^{\mathrm{Sh}}$.

Conversely, it is clear that for any formula $\phi(x;y)$ and any $b \in N^y$,
the basic $\mathcal{L}^{\mathrm{Sh}}$-formula $P_{\phi,b}(x)$ is externally formula-definable by $\phi(x;b)$.
Any continuous connectives (not quantifiers) we apply to these predicate symbols will preserve external formula-definability, if we apply them to the defining formulas, so by induction, all quantifier-free $\mathcal{L}^{\mathrm{Sh}}$-formulas are $\mathcal{L}$-externally formula-definable.

The externally definable predicates are exactly the uniform limits of externally formula-definable predicates,
as the uniform limit of $(\phi_n(x;b_n): n < \omega)$ can be externally defined with $\lim_{n}\phi_n(x;b_0,b_1,\dots)$,
with $b_0b_1\dots$ a tuple over $N$. Thus they are exactly the uniform limits of quantifier-free $\mathcal{L}^{\mathrm{Sh}}$-formulas, which are the quantifier-free $\mathcal{L}^{\mathrm{Sh}}$-definable predicates.
\end{proof}

For the remainder of this section, we assume $T$ is NIP, and fix $M \preceq N$ $|M|^+$-saturated.

\begin{lem}\label{lem_sh_hd}
Let $f : M^x \to [0,1]$ be externally definable.
Then there is a definable predicate $\phi(x;b)$ with $b \in \mathcal{U}^y$ such that $\phi(a;b) = f(a)$ for all $a \in M^x$ and
for every $M$-definable predicate $\theta(x;c)$ with $\theta(a;c) \leq f(a)$ for all $a \in M^x$,
we also have $\mathcal{U} \vDash \theta(x;c) \leq \phi(x;b)$.
\end{lem}
\begin{proof}
Let $\psi(x;d)$ be an external definition of $f$, with $M \preceq N$ and $d \in N^z$.
Then consider the pair $(N,M)$, and apply Theorem \ref{thm_UHD}.
There is some elementary extension $(N,M)\preceq (N',M')$ and an honest definition $\phi(x;b)$ of $\psi(x;d)$ over $M$ with $b \in M'^y$.
This means that for $a \in M^x$, $\phi(a;b) = \psi(a;d) = f(a)$, and $(N',M')\vDash \sup_{x \in P} \psi(a;d) \dot{-} \phi(a;b) = 0$.
Now let $\theta(x;c)$ with $c \in M^w$ be such that $\theta(a;c) \leq f(a) = \psi(a;d)$ for all $a \in M^x$
Then $(N,M) \vDash \sup_{x \in P} \theta(x;c) \dot{-} \psi(a;d) = 0$, so the same condition holds in $(N',M')$, and thus
$(N',M')\vDash \sup_{x \in P} \theta(x;c) \dot{-} \phi(a;b) = 0$, so
$M' \vDash \sup_x \theta(x;c) \dot{-}\phi(a;b) = 0$,
and by elementarity, $\mathcal{U} \vDash \sup_x \theta(x;c) \dot{-}\phi(a;b) = 0$.
\end{proof}

We now generalize Shelah's expansion theorem to continuous logic, using honest definitions as in the proof in the discrete case given in \cite{cs1}.
\begin{thm}\label{thm_sh_qe}
The structure $M^{\mathrm{Sh}}$ admits quantifier elimination.
\end{thm}
\begin{proof}
By \cite[Lemma 13.5]{mtfms}, it suffices to show that if $\phi(x;y)$ is a quantifier-free $\mathcal{L}^{\mathrm{Sh}}$-formula, then
$\inf_x \phi(x;y)$ is approximable by quantifier-free formulas, that is, is a quantifier-free $\mathcal{L}^{\mathrm{Sh}}$-definable predicate.
By Lemma \ref{lem_sh_predicate}, that means that it is enough to show that if $f(x,y) : M^{xy} \to [0,1]$ is externally formula-definable,
then $\inf_{x \in M}f(x,y)$ is also externally definable.

Let $f(x,y)$ be externally formula-definable.
In particular, there is some constant $C$ such that $f(x,y)$ is $C$-Lipschitz, and $f(x,y)$ is externally definable.
By Lemma \ref{lem_sh_hd}, we may assume that $f(x;y)$ is given by a $\mathcal{L}$-predicate $\phi(x,y;d)$ with $d \in \mathcal{U}^z$,
such that for every $\mathcal{L}(M)$-definable predicate $\theta(x,y;c)$ with $\theta(a,b;c) \leq f(a,b)$ for all $a,b \in M^{xy}$,
we also have $\mathcal{U} \vDash \theta(x,y;c) \leq \phi(x,y;d)$.
We claim that $\inf_{x \in M} f(x;y)$ is externally definable by $\inf_x \phi(x,y;d)$.
Clearly for any $b \in M^y$, 
$$\inf_x \phi(x,b;d) = \inf_{x \in \mathcal{U}} \phi(x,b;d) \leq \inf_{x \in M} \phi(x,b;d) = \inf_{x \in M} f(x,b),$$
so it suffices to show that for $b \in M^y$,
$\inf_{x \in M} f(x,b) \leq \inf_{x \in \mathcal{U}} \phi(x,b;d)$.

Let $\zeta(x,y)$ be the $\mathcal{L}(M)$-formula $\inf_{x \in M}f(x;b) - C d(y,b)$, noting that the infimum $\inf_{x \in M}f(x;b)$ is just a constant.
Then for all $(a',b') \in M^{xy}$, we find that by the Lipschitz property of $f$,
$$f(a',b') \geq f(a',b) - Cd(b',b) \geq \inf_{x \in M} f(x;b) - Cd(b',b) = \zeta(a',b').$$
Thus by assumption on $\phi$, $\mathcal{U} \vDash \zeta(x,y) \leq \phi(x,y;d)$, so $\mathcal{U} \vDash \inf_x\zeta(x,b) \leq \inf_x\phi(x,b;d)$.
However, $\zeta$ has no dependence on $x$, so
$$\inf_{x \in \mathcal{U}}\zeta(x,b) = \inf_{x \in M} f(x;b) - Cd(b,b) = \inf_{x \in M} f(x;b),$$
and thus $\inf_{x \in M} f(x;b) \leq \inf_{x \in \mathcal{U}} \phi(x,b;d)$.
\end{proof}

\begin{cor}\label{cor_sh_predicate}
    The predicates $M^x \to [0,1]$ given by formulas $\phi(x)$ in $\mathcal{L}^{\mathrm{Sh}}$
    are exactly the externally formula-definable predicates on $M$,
    and the $\mathcal{L}^{\mathrm{Sh}}$-definable predicates on $M^{\mathrm{Sh}}$ are precisely the externally definable predicates on $M$.
\end{cor}
\begin{proof}
By \ref{thm_sh_qe}, we can drop the ``quantifier-free'' descriptions from Lemma \ref{lem_sh_predicate}.
\end{proof}

\begin{cor}\label{cor_sh_nip}
The structure $M^{\mathrm{Sh}}$ is NIP.
\end{cor}
\begin{proof}
    Any definable predicate over $M^{\mathrm{Sh}}$ corresponds to an externally definable predicate $\phi(x;b)$ over $M$, which is dependent.
\end{proof}

\section{Definitions of Distality}\label{sec_distal}
Let $T$ be a theory in continuous logic. We will present several possible definitions of distality, and determine which of them are equivalent.

The first definition, in terms of indiscernible sequences, is unchanged from discrete logic.
\begin{defn}[Distality]\label{indiscernible}
    Let $I$ be an indiscernible sequence. Then we say that $I$ is \emph{distal} when for any indiscernible sequence $I_1 + I_2$ with the same EM-type as $I$, where $I_1$ and $I_2$ are dense and without endpoints,
    if $I_1 + d + I_2$ is also indiscernible and $I_1 + I_2$ is indiscernible over a set $B$, then $I_1 + d + I_2$ is also indiscernible over $B$.

    We say $T$ is \emph{distal} when every indiscernible sequence in a model of $T$ is distal.
\end{defn}
This definition also appears in a limited continuous context in \cite{kp21}.
Note that we could equivalently add parameters to this definition. If $I + d + J$ is indiscernible over $A$ with $I + J$ indiscernible over $AB$, then if $I + d + J$ is not indiscernible over $AB$, there must be finite tuples $a \subseteq A$, $b \subseteq B$ such that $I + d + J$ is not indiscernible over $ab$. If we let $I_a = (ia : i \in I)$ and $J_a = (ja : j \in J)$, then $I_a + da + J_a$ will be indiscernible over $\emptyset$ but not over $b$, and $I_a + J_a$ will be indiscernible over $b$, contradicting distality.

First we check that this definition of distality implies NIP.
\begin{thm}
    If a metric theory $T$ is distal, then $T$ is NIP.
    \end{thm}
    \begin{proof}
    Assume $T$ is not NIP. Let $(a_i : i \in \omega)$ be an indiscernible sequence, $b$ a tuple, $\phi(x;y)$ a formula, and $0 \leq r < s \leq 1$ such that $\vDash \phi(a_i;b)\leq r$ when $i$ is even and $\vDash \phi(a_i;b)\geq s$ when $i$ is odd.
    
    We claim that there are sequences $I, J$ of order type $\Q$ and some $d$ such that $I + d + J$ is indiscernible, $I + J$ is indiscernible over $b$, but for all $i \in I + J$, $\vDash \phi(i;b) \leq r$ while $\vDash \phi(d;b)\geq s$.
    If so, this will contradict distality.
    Such an $I + d + J$ is exactly a realization of the following partial type $\Sigma$ in variables 
    $$X = X_I \cup \{x_d\} \cup X_J = \{x_{iq} : q \in \Q\} \cup \{x_d\} \cup \{x_{jq} : q \in \Q\},$$
    where $X^{n}_<$ is the set of increasing $n$-tuples of $X$, and $(X_I \cup X_J)^n_<$ is defined similarly:
    \begin{align*}T &\cup \{|\psi(\bar x) - \psi(\bar x')|  = 0 : \psi \in \mathcal{L}; \bar x, \bar x' \in X_<^n\}\\
    &\cup \{|\psi(\bar x,b) - \psi(\bar x',b)| \leq \frac{1}{m} : \psi \in \mathcal{L}; \bar x, \bar x' \in (X_I \cup X_J)^n_<; m \in \N\}\\
    &\cup \{\phi(x,b) \leq r: x \in X_I \cup X_J\}\\
    &\cup \{\phi(x_d,b) \geq s\}.
    \end{align*}
    
    It suffices to show that $\Sigma$ is consistent. Let $\Sigma_0 \subset \Sigma$ be finite, and let $\bar x \in (X_I)^n_<$, $\bar x' \in (X_J)^n_<$, and $x_d$ include all the variables of $X$ appearing in $\Sigma_0$.
    Then we will find a finite subsequence of $(a_i : i \in \omega)$ realizing $\Sigma_0$.
    It will automatically be $\emptyset$-indiscernible, and we will interpret $\bar x, \bar x'$ with even elements of the sequence, and $x_d$ with an odd element, so we need only make sure that a finite set of conditions of the form $|\psi(x_1,\dots, x_n,b) - \psi(x_1',\dots, x_n',b)| \leq \frac{1}{m}$ are satisfied. 
    
    To do this, we find an infinite subsequence of $(a_{2i} : i \in \omega)$ such that for all $\psi$ in a finite set $\Psi_0 = \{\psi_0,\dots, \psi_r\}$, some fixed $m$,a
    and each pair of increasing $n$-tuples $\bar a,\bar a'$, we have
    $|\psi(\bar a,b) - \psi(\bar a',b)| \leq \frac{1}{m}$.
    Assume for induction that $S$ is an infinite subsequence such that this holds for all $\psi_i$ with $i < k$. (For $k = 0$, we set $S = (a_{2i} : i \in \omega)$.)
    Then we color all finite subsequences $x_1 < \dots < x_n$ of $S$ with $m$ colors, assigning a tuple color $c_j$ when $\frac{j}{m} \leq \psi_k(x_1,\dots,x_n,b) < \frac{j + 1}{m}$.
    By Ramsey's Theorem, there must be an infinite monochromatic subsequence, which satisfies the induction step.
    
    Once we have this infinite subsequence $S$, we can select $\bar a$ to be an arbitrary increasing subsequence of $S$.
    Then we interpret $x_d$ with some $a_{2i+1}$ greater than all of $\bar a$, and let $\bar a'$ be in $S$ and greater than $a_{2i + 1}$.
    \end{proof}

We will now show some useful lemmas for showing that indiscernible sequences are distal.
\begin{lem}[{Generalizes \cite[Lemma 2.7]{distal_simon}}]\label{lem_distal_seq}
Assume $T$ is NIP. If $I$ is a dense indiscernible sequence without endpoints, then $I$ is distal if and only if for every partition
$I = I_1 + I_2 + I_3$ where $I_1, I_2, I_3$ have no endpoints, then for all $b_1, b_2$ such that
$I_1 + b_1 + I_2 + I_3$ and $I_1 + I_2 + b_2 + I_3$ are indiscernible, then $I_1 + b_1 + I_2 + b_2 + I_3$ is also.
\end{lem}
\begin{proof}
Clearly distality implies this condition, so it suffices to check that such a sequence is distal.

First we observe that this alternative characterization of distality (at least for dense sequences) only depends on the EM-type of $I$.
There exist $b_1, b_2$ such that $I_1 + b_1 + I_2 + I_3$ and $I_1 + I_2 + b_2 + I_3$ are indiscernible, but $I_1 + b_1 + I_2 + b_2 + I_3$ is not,
if and only if there exists some formula $\phi(y_1,x_1,y_2,x_2,y_3)$, an $\varepsilon > 0$, such that
when $(y_1,x_1,y_2,x_2,y_3)$ is an increasing tuple of variables,
$\phi(y_1,x_1,y_2,x_2,y_3) = 0$ is in the EM-type of $I$, but
$\phi(y_1,x_1,y_2,x_2,y_3) =\varepsilon$ is consistent with $(y_1,x_1,y_2,y_3)$ and $(y_1,y_2,x_2,y_3)$ satisfying the EM-type of $I$.

Now we will show another property that follows from this condition: for all natural numbers $n$, if $I = I_0 + I_1 + \dots I_n$ is a partition into dense endpointless pieces,
and $b_0,\dots,b_{n - 1}$ are such that for each $i$, $I_0 + \dots + I_i + b_i + I_{i + 1} + \dots + I_n$ is indiscernible, then 
$I_0 + b_0 + I_1 + b_1 + \dots + b_{n - 1} + I_n$ is also. We proceed by induction on $n$, with cases $n = 0,1$ trivial, and case $n = 2$ assumed.
Assuming this works for $n$ for all such sequences, partition or sequence as $I_0 + I_1 + \dots + I_{n + 1}$, and find suitable $b_0,\dots,b_{n}$.
Then as $I' = I_0 + b_0 + I_1 + I_2 + \dots + I_n$ is indiscernible, it has the same EM-type as $I$, so it also has this property.
Thus the sequence obtained by inserting $b_i$ into $I'$ is indiscernible for all $i > 0$, so by our induction hypothesis,
inserting all $n$ extra elements gives an indiscernible sequence, as desired.

If our sequence $I$ is not distal, then there exists a set $B$, a tuple $d$, and sequences $I_1 + I_2$ indiscernible over $B$, with the same EM-type as $I$,
where $I_1$ and $I_2$ are dense and without endpoints, and $I_1 + d + I_2$ is indiscernible but not indiscernible over $B$. 

Thus there is some formula $\phi(x_1,x,x_2)$ with parameters in $B$, and finite tuples $i_1 \subseteq I_1$ and $i_2 \subseteq I_2$ such that
for any $i \in I_1 + I_2$ between $i_1$ and $i_2$,
$\phi(i_1,i,i_2) = 0$, but $\phi(i_1,d,i_2) = \varepsilon > 0$.
By avoiding $i_1$ and $i_2$, we can find a final segment $I_1' \subseteq I_1$ and an initial segment $I_2' \subseteq I_2$
such that $I_1' + I_2'$ is indiscernible over $Bi_1i_2$.
By $Bi_1i_2$-indiscernibility, we see that for any partition of $I_1' + I_2'$ into endpointless pieces,
there is some element $d'$ that could be inserted, maintaining indiscernibility, but with $\phi(i_1,d',i_2) = \varepsilon$.

Now partition $I_1' + I_2'$ into a countable infinite sequence $J_0 + J_1 + J_2 + \dots$ of endpointless parts.
For each $n \in \N$, there is $d_n$ such that inserting $d_n$ between $J_n$ and $J_{n+1}$ maintains indiscernibility, but 
$\phi(i_1,d_n,i_2) = \varepsilon$. Inserting all of these either violates indiscernibility or NIP, as $\phi(i_1,d_n,i_2)$ alternates infinitely often between 0 and $\varepsilon$.
We have shown that for each $n$, inserting all of $d_0,\dots, d_n$ maintains indiscernibility, so inserting each $d_n$ at once maintains indiscernibility.
Thus NIP fails, contradicting our hypothesis.
\end{proof}

This lemma is the metric version of a special case of \cite[{Lemma 2.8}]{distal_simon}, on strong base change.
It is particularly useful in conjunction with Lemma \ref{lem_distal_seq}.
\begin{lem}\label{lem_base_change}
Let $I = I_0 + I_1 + I_2$ be an indiscernible sequence, with $A \supset I$ a set of parameters,
such that $I_0, I_1, I_2$ are dense without endpoints.
Let $a$ and $b$ be such that $I_0 + a + I_1 + I_2$ and $I_0 + I_1 + b + I_2$ are indiscernible.
Then there are $a'$ and $b'$ such that $\mathrm{tp}(a'b'/I) = \mathrm{tp}(ab/I)$, 
$\mathrm{tp}(a'/A) = \lim(I_0/A)$ and $\mathrm{tp}(b'/A) = \lim(I_1/A)$.
\end{lem}
\begin{proof}
Assume that the conclusion is false.
Then by compactness, there are closed conditions $\phi(x,y) = 0 \in \mathrm{tp}(ab/I)$,
$\psi_0(x) = 0 \in \lim(I_0/A)$ and $\psi_1(y) = 0 \in \lim(I_1/A)$
such that $\{\phi(x,y) = 0, \psi_0(x) = 0, \psi_1(y) = 0\}$ is inconsistent.
There is some minimum value $\varepsilon$ taken by $\max(\psi_0(x),\psi_1(y))$ on the set of all types in $S_{xy}(A)$
satisfying $\phi(x,y) = 0$, and we see that $\varepsilon > 0$.
Let $I_\phi \subset I$ be a finite tuple containing all parameters of $\phi$.

Because $\psi_0(x) = 0 \in \lim(I_0/A)$, we can find a final segment $J_{0-} \subseteq I_0$
such that $\psi_0(x) \leq \frac{\varepsilon}{2}$ on all of $J_{0-}$, and an initial segment $J_{0+}$ of $I_1$ such that
$J_{0-} + J_{0+}$ lies in the space between elements of $I_\phi$.
We can also find $J_{1-} \subseteq I_1, J_{1+} \subseteq I_2$ satisfying the same properties for $\psi_1$.
As $J_{0-} + J_{0+}$ and $J_{1-} + J_{1+}$ lie between elements of $I_\phi$, these sequences are mutually indiscernible over $I_\phi$.
As $a$ and $b$ also lie in those intervals, we find that for any $a' \in J_{0-} + J_{0+}$ and $b' \in J_{1-} + J_{1+}$,
$\phi(a',b') = 0$. This means that there exist $e_0,e_1$ such that 
$I_0 + e_0 + I_1 + I_2$ and $I_0 + I_1 + e_1 + I_2$ are indiscernible and $\phi(e_0,e_1) = 0$.
Thus for $i = 0$ or $i = 1$, $\psi_i(e_i)\geq \varepsilon$.
We now add that value of $e_i$ into the sequence, maintaining indiscernibility, and repartition.

For the sake of simplicity, assume that $e_1$ is the added value.
Then we repartition $J_{1-} + e_1 + J_{1+}$ as $J_{1-}' + J_{1+}'$, where $J_{1-}'$ is a strict initial segment of $J_{1-}$.
We repeat the earlier process, finding $e_0', e_1'$ such that $J_{0-} + e_0 + J_{0+}$ and $J_{1-}' + J_{1+}'$ remain mutually indiscernible over $I_0$, as do
$J_{0-} + J_{0+}$ and $J_{1-}' + e_1' + J_{1+}'$, while maintaining $\phi(e_0', e_1') = 0$.
Thus we add either $e_0'$ or $e_1'$, and repeat infinitely many times. 

In conclusion, we have added infinitely many points to either $J_{0-}$ or $J_{1-}$. Assume without loss of generality it was $J_{1-}$.
Then we have an indiscernible sequence consisting of $J_{1-}$ and the added points where the value of $\psi_{1}$ alternates infinitely many times between being $\psi_1(y)\leq \frac{\varepsilon}{2}$,
as on all original values of $J_{1-}$, and $\psi_1(y)\geq \varepsilon$, as on all the new added points.
This contradicts NIP.
\end{proof}

In the rest of this section, we will generalize several other definitions of distality, in terms of types and formulas, to continuous logic.
We will check that these are the correct generalizations by showing that these definitions are all equivalent to our first definition in terms of indiscernible sequences.
\begin{thm}\label{thm_distal}
If a metric theory $T$ is NIP, then the following are equivalent:
\begin{enumerate}
    \item $T$ is distal.
    \item Every global type is distal.
    \item Every formula admits strong honest definitions.
    \item Every formula admits an $\varepsilon$-distal cell decomposition for each $\varepsilon > 0$.
\end{enumerate}
\end{thm}
We will prove this over the following subsections by showing that $1 \implies 2, 2 \implies 3,$ $3 \implies 4$, and $4 \implies 1$,
introducing the definitions of distal types (Definition \ref{def_distal_type}),
strong honest definitions (Definition \ref{defn_SHD}),
and distal cell decompositions (Definition \ref{defn_DCD}) as we go.

\subsection{Distal Types}
We now restate the definition of distal types in an NIP theory, which also works as-is in the continuous context.
\begin{defn}[{Distal types, \cite[Def. 9.3]{nip_guide}}]\label{def_distal_type}
Assume $T$ is NIP.
Let $p$ be a global $A$-invariant type. Then $p$ is \emph{distal} over $A$ when for any tuple $b$, if $I \vDash p^{(\omega)} |_{Ab}$, then $p|_{AI}$ and $\mathrm{tp}(b/AI)$ are weakly orthogonal.
(That means that if we write $q(y) = \mathrm{tp}(b/AI)$, there is a unique complete type over $A$ extending $p(x) \cup q(y)$.)

If $p$ is distal over all $A$ such that $p$ is invariant over $A$, then we just say that $p$ is distal, without specifying $A$.
\end{defn}

\begin{thm}\label{thm_types_are_distal}
In a distal theory, all invariant types are distal.
\end{thm}
\begin{proof}
Let $p$ be a global $A$-invariant type, let $b$ be a tuple, and let $I \vDash p^{(\omega)}|_{Ab}$. We wish to show that $p|_{AI}$ is weakly orthogonal to $q(y) = \mathrm{tp}(b/AI)$. One such type is $\mathrm{tp}(a_pb/AI)$ for any $a_p \models p|_{AIb}$, so for contradiction, assume there is some $a \models p|_{AI}$ such that $a \not\models p_{AIb}$. We then construct another Morley sequence. Let $J \models p^{(\omega)}|_{MIa}$. Then $I + a + J \models p^{(\omega + \omega)}|_A$, and is thus indiscernible over $A$, while $I + J \models p^{(\omega + \omega)}|_M$, and is thus indiscernible over $Ab \subseteq M$. For any $j \in J$, $j \models p|_{AIb}$, but $a \not\models p|_{AIb}$, so $I + a + J$ is not indiscernible over $Ab$, contradicting distality.
\end{proof}

\subsection{Strong Honest Definitions}
We will now prove a series of versions of strong honest definitions. As with honest definitions, we start by assuming distality to show a version expressed in terms of pairs, derive a finitary version expressible without pairs, uniformize that finitary version using the $(p,q)$-theorem, and then prove distality from strong honest definitions, showing that all of these statements are equivalent.

There will be two basic ways to express strong honest definitions.
The first is the continuous version of the version from \cite[Prop. 19]{cs2}.
\begin{defn}\label{defn_SHD}
    Let $A$ be a closed subset of $M^y$ where $M \preceq \mathcal{U}$ and $(M,A)\preceq (M',A')$.
    Let $\phi(x;y)$ be a definable predicate, let $a \in M$, and let $\theta(x;d)$ be an $A'$-predicate.
    We say that $\theta(x;d)$ is a \emph{strong honest definition} for $\phi(a;y)$ over $A$ when 
    \begin{itemize}
        \item $M' \vDash \theta(a;d) = 0$
        \item For all $a' \in M'^x$, $b \in A$, $|\phi(a';b) - \phi(a;b)|\leq \theta(a';d)$.
    \end{itemize}

    For either of these definitions, if the same predicate $\theta(x;z)$ works for any choice of $M,A,b$, then we call
    $\theta(x;z)$ a strong honest definition for $\phi(x;y)$.
\end{defn}
Essentially, $\theta(x;d)$ controls how much the type $\mathrm{tp}_{\phi}(x/A)$ differs from $\mathrm{tp}_{\phi}(a/A)$.
In classical logic, when $\phi$ and $\theta$ only take values 0 and 1 corresponding to true and false,
this definition is equivalent to $M' \vDash \theta(a;d)$ and $\theta(x;d) \vdash \mathrm{tp}_{\phi}(a/A)$.
This is precisely the presentation of strong honest definitions in \cite[Proposition 19]{cs2}.
We see that as in classical logic, strong honest definitions always exist in distal theories.

\begin{thm}\label{thm_SHD}
    Assume $T$ is distal. Let $M \models T$, $A \subseteq M$ closed,
    $\phi(x;y)$ a definable predicate, and $a \in M^x$.
    There is some elementary extension $(M,A) \preceq (M',A')$ such that
    $\phi(a;x)$ admits a strong honest definition $\theta(x;d)$ with $d \in A'^z$.
\end{thm}
\begin{proof}
    As before, let $S_A \subseteq S_y(\mathcal{U})$ be the set of global types approximately realized in $A$.
    We will show that $\mathrm{tp}(a/A') \times S_A|_{A'} \vDash \phi(x;y) = \phi(a;y)$, and then extract the strong honest definition from there.
    
    To do this, let $p(y) \in S_A$ be a global type. 
    We claim that there is $b \in A'$ realizing $p$ over $MB$ for any small $B \subseteq A'$.
    By the saturation of $\mathcal{M}'$, it suffices to show that the type $p(y)|_{MB} \cup \{P(y) = 0\}$ is consistent.
    For this, it is enough to show that for every condition $\pi(y) = 0 \in p(y)|_{MB}$, and every $\varepsilon > 0$, $\pi(y)\leq \varepsilon$ is consistent with $P(y) = 0$.
    As $[\pi(y) < \varepsilon]$ is an open set containing $p(y)$, it must also intersect the set of realizations of $A$, and thus intersects $[P(y) = 0]$, so $\pi(y)\leq \varepsilon$ is consistent with $P(y) = 0$.

    This allows us to construct a Morley sequence $I$ for $p$ over $M$ in $A'$, by recursively defining $a_n$ to be an element of $A'$ realizing $p|_{Ma_0\dots a_{n-1}}$.
    By Theorem \ref{thm_types_are_distal}, for any $p(y) \in S_A$, $p|_{AI}$ is weakly orthogonal to $\mathrm{tp}(a/AI)$,
    so $\mathrm{tp}(a/AI) \times p|_{AI} \vDash \phi(x;y) = \phi(a;y)$, and expanding the parameter sets, we see that $\mathrm{tp}(a/A') \times p|_{A'} \vDash \phi(x;y) = \phi(a;y)$.
    As this holds for all $p \in S_A$, the condition $\phi(x;y) = \phi(a;y)$ holds everywhere on $\mathrm{tp}(a/A') \times S_A|_{A'}$, so the predicate $|\phi(x;y) - \phi(a;y)|$ is zero.

    We now apply Lemma \ref{lem_dominate} to the partial $A'$ types $\mathrm{tp}(a/A')$ and $S_A|_{A'}(y)$ on $(x,y)$ and the predicate $|\phi(x;y) - \phi(a;y)|$,
    and find a definable predicate $\theta(x;d)$ with $d \in A'$ such that $\mathrm{tp}(a/A')$ implies $\theta(x;d) = 0$
    and for all $b$ satisfying a type in $S_A|_{A'}(y)$, $|\phi(x;b) - \phi(a;b)| \leq \theta(x;d)$.
    In particular, for all $b \in A$, $|\phi(x;b) - \phi(a;b)| \leq \theta(x;d)$.
\end{proof}

There is another form of strong honest definitions, which is literally an honest definition in the sense of Definition \ref{defn_HD}.
We call these ``strong$^*$ honest definitions,'' as their existence is related to existence of strong honest definitions for the dual predicate.
\begin{defn}\label{defn_SHD_as_HD}
    Let $A$ be a closed subset of $M^x$ where $M \preceq \mathcal{U}$ and $(M,A)\preceq (M',A')$.
    Let $\phi(x;b)$ be an $M$-predicate, and let $\psi(x;d)$ be an $A'$-predicate.
    We say that $\psi(x;d)$ is a \emph{strong$^*$ honest definition} for $\phi(x;b)$ over $A$ when 
    \begin{itemize}
        \item for all $a \in A$, $\phi(a;b) = \psi(a;d)$
        \item for all $a \in M'^x$, $\phi(a;b) \leq \psi(a;d)$.
    \end{itemize}
    
    If the same predicate $\psi(x;z)$ works for any choice of $M,A,b$, then we call $\psi(x;z)$ a strong$^*$ honest definition for $\phi(x;y)$.
    
    For all $\phi(x;y)$ and $\psi(y;z)$, we also define a predicate 
    $$\mathrm{SHD}_{\phi,\psi,P}(y;z) = \max\left(\sup_{x : P(x)}|\phi(x;y) - \psi(x;z)|,\sup_x\phi(x;y) \dot{-} \psi(x;z)\right).$$
    Then for $d \in A'^z$, $(M',A) \vDash \mathrm{SHD}_{\phi,\psi,P}(b;d)$ if and only if $\psi(x;d)$ is a strong$^*$ honest definition for $\phi(x;b)$.
    We will abuse notation later to write $\mathrm{SHD}_{\phi,\psi,A}(b;d)$ for the value of $\mathrm{SHD}_{\phi,\psi,P}(b;d)$ in $(M',A)$.
\end{defn}

We see that strong honest definitions imply the existence of strong$^*$ honest definitions for the dual predicate.
\begin{lem}\label{lem_SHD_as_HD}
    Let $A$ be a closed subset of $M^x$ where $M \preceq \mathcal{U}$ and $(M,A)\preceq (M',A')$.
    Let $\phi(x;b)$ be an $M$-predicate.
    If $\phi^*(b;x)$ admits a strong honest definition $\theta(y;d)$ over $A$, then $\phi(x;b)$ admits a strong$^*$ honest definition $\psi(x;d)$ over $A$,
    with the same parameters, and $\psi(x;z)$ depending only on $\theta(y;z)$.
\end{lem}
\begin{proof}
Let $\psi(x;d) = \sup_y (\phi(x;y)\dot{-}\theta(y;d))$.
Thus for $a \in M'^x$, $\psi(a;d) = \sup_y (\phi(a;y)\dot{-}\theta(y;d))$.
By plugging in $y = b$, we see that $$\psi(a;d) \geq \phi(a;b)\dot{-}\theta(b;d) = \phi(a;b).$$

Now let $a \in A$.
For all $b' \in M'^y$, we have $|\phi(a;b') - \phi(a;b)|\leq \theta(b';d)$, so $\phi(a;b') \dot{-}\theta(b';d) \leq \phi(a;b)$,
and thus $\psi(a;d) \leq \phi(a;b)$, so $\phi(a;b) = \psi(a;d)$.
\end{proof}

We can recover strong honest definitions from strong$^*$ honest definitions for both the original predicate and its complement.
\begin{lem}\label{lem_SHD_as_HD2}
    If $\phi(x;b)$ and $1 - \phi(x;b)$ admit strong$^*$ honest definitions over $A$ then $\phi^*(b;x)$ admits a strong honest definition over $A$.
\end{lem}
\begin{proof}
Assume that $\phi(x;b)$ admits a strong$^*$ honest definition $\psi^+(x;d)$ over $A$,
and $1 - \phi(x;b)$ admits a strong$^*$ honest definition $\psi'(x;d)$ over $A$.
Then by setting $\psi^-(x;d) = 1 - \psi'(x;d)$, we find that
\begin{itemize}
    \item for all $a \in A$, $\phi(a;b) = \psi^-(a;d) = \psi^+(a;d)$
    \item for all $a \in M'^x$, $\psi^-(a;d) \leq \phi(a;b) \leq \psi^+(a;d)$.
\end{itemize}

Then we let $\theta(y;z) = \sup_x \max(\psi^-(x;z)\dot{-}\phi(x;y), \phi(x;y)\dot{-}\psi^+(x;z))$.
For every $a \in M'^x$, we have that 
$$\psi^-(a;d)\dot{-}\phi(a;b) = \phi(a;b)\dot{-}\psi^+(a;d) = 0,$$
so
$$\theta(b;d) = \sup_x \max(\psi^-(x;d)\dot{-}\phi(x;b), \phi(x;b)\dot{-}\psi^+(x;d)) = 0.$$

Now let $a \in A, b' \in M'^y$.
We have that 
\begin{align*}
    |\phi(a;b)-\phi(a;b')|
&= \max(\phi(a;b)\dot{-}\phi(a;b'),\phi(a;b')\dot{-}\phi(a;b)) \\
&= \max(\psi^-(a;d)\dot{-}\phi(a;b'), \phi(a;b')\dot{-}\psi^+(a;d)) \\
&\leq \theta(b';d).
\end{align*}
\end{proof}

As with honest definitions, we can take forced limits of approximate strong$^*$ honest definitions to get strong$^*$ honest definitions, and the proof is essentially the same.
\begin{lem}\label{lem_SHD_lim}
    Let $A$ be a closed subset of $M^y$ where $M \preceq \mathcal{U}$, and let $\phi(x;b)$ be an $M$-predicate.
    Let $(M,A)\preceq (M',A')$, and let $d \in A'^z$.
    Let $(\psi_n(x;z): n \in \N)$ be a sequence of definable predicates with
    $\mathrm{SHD}_{\phi,\psi_n,A}(b;d) \leq 2^{-n}$ for each $n$.
    Then $\mathcal{F}\lim\psi_n(x;d)$ is a strong$^*$ honest definition for $\phi(x;b)$ over $A$.

    If instead we have a sequence $(\psi_n(x;z_n): n \in \N)$ with different $d_n \in A'^{z_n}$ for each $n$ such that
    $\mathrm{SHD}_{\phi,\psi_n,A}(b;d_n) \leq 2^{-n}$,
    then $\mathcal{F}\lim\psi_n(x;d)$ is a strong$^*$ honest definition for $\phi(x;b)$ over $A$,
    where $d$ is a concatenation of all the tuples $d_n$.
\end{lem}

We now deduce a finitary version of strong$^*$ honest definitions, without having to introduce an elementary extension.
The proof is analogous to the proof of \ref{lem_FHD}.

\begin{lem}\label{lem_FSHD_as_HD}
    Let $A$ be a closed subset of $M^y$ where $M \preceq \mathcal{U}$, and let $\phi(x;b)$ be an $M$-predicate.
Fix $(M,A) \preceq (M',A')$ to be $|M|^+$-saturated, $\varepsilon > 0$, and a definable predicate $\psi(x;z)$.

If there exists $d \in A'^z$ such that $\mathrm{SHD}_{\phi,\psi,A}(b;d) < \varepsilon$, then for all finite $A_0 \subseteq A$,
we have there is a tuple $d_{A_0} \in A^z$ such that $\mathrm{SHD}_{\phi,\psi,A_0}(b;d_{A_0}) < \varepsilon$.

Conversely, if for all finite $A_0 \subseteq A$, there is a tuple $d_{A_0} \in A^z$ such that $\mathrm{SHD}_{\phi,\psi,A_0}(b;d_{A_0}) \leq \varepsilon$,
then there exists $d \in A'^z$ such that $\mathrm{SHD}_{\phi,\psi,A}(b;d) \leq \varepsilon$.
\end{lem}

We can now uniformize strong honest definitions using Lemma \ref{lem_FSHD_as_HD}.
The same argument used to prove \ref{lem_UHD} and then \ref{thm_UHD} applies again:

\begin{thm}\label{thm_USHD_as_HD}
Assume $T$ is distal.
Every definable predicate $\phi(x;y)$ admits a strong$^*$ honest definition $\psi(x;z)$. That is,
Then there is a definable predicate $\psi(x;z)$ such that for any $\mathcal{M} \models T$, $A \subseteq M$ closed with $|A| \geq 2$, and $b \in M^y$ with $|A|\geq 2$,
there is some $d$ such that $\psi(x;d)$ is a strong$^*$ honest definition for $\phi(x;y)$ over $A$.
\end{thm}

Finally, by \ref{lem_SHD_as_HD2}, we can translate this back into a uniformized version of strong honest definitions.
\begin{thm}\label{thm_USHD}
    Assume $T$ is distal.
    Every definable predicate $\phi(x;y)$ admits a strong honest definition $\theta(y;z)$.
\end{thm}

\subsection{Distal Cell Decompositions}
While our finitary approximation to a strong$^*$ honest definition matches our notions for honest definitions, the finitary approximation to strong honest definitions will more closely resemble our approach to UDTFS.
As we will use these for more combinatorial applications, we will use the conventions of distal cell decompositions from \cite{cgs}.

\begin{defn}\label{defn_DCD}
    Let $\phi(x;y)$ be a definable predicate, let $\Psi$ be a finite set of definable predicates of the form $\psi(x;y_1,\dots,y_k)$, where $k$ is finite.
    
    We say that $\Psi$ \emph{weakly defines} a $\varepsilon$-\emph{distal cell decomposition} over $M$ for $\phi(x;y)$
    when for every finite $B \subseteq M^y$ with $|B| \geq 2$,
    there are sets $B_\psi \subseteq B$ for each $\psi \in \Psi$ such that
    the predicate $\sum_{\psi \in \Psi}\sum_{\bar b \in B_\psi} \psi(x;\bar b)$ is always nonzero,
    and for each $\psi \in \Psi, \bar b \in B_\psi$ and $b \in B$, we have the bound
    $$\sup_{x,x'} \min(\psi(x;\bar b),\psi(x';\bar b), |\phi(x;b) - \phi(x';b)|\dot{-}\varepsilon)  = 0,$$
    indicating that for all $a,a'$ in the support of $\psi(x;\bar b)$, $|\phi(a;b) - \phi(a';b)|\leq\varepsilon$.

    Let $\Theta = \{\theta_\psi : \psi \in \Psi\}$ where for each $\psi(x;y_1,\dots,y_k) \in \Psi$, 
    $\theta_\psi$ is a definable predicate of the form $\theta(y;y_1,\dots,y_k)$.

    We say that $\Psi$ and $\Theta$ \emph{define} a $\varepsilon$-\emph{distal cell decomposition} over $M$ for $\phi(x;y)$
    when for every finite $B \subseteq M^y$ with $|B| \geq 2$,
    we may let $B_\psi = \{(b_1,\dots,b_k) \in B^k : \forall b \in B, \theta_\psi(b;b_1,\dots,b_k) = 0\}$
    in the above definition.
\end{defn}
To recover the classical logic definition from \cite{cgs}, we may choose any $0 < \varepsilon < 1$ and let $\phi = 0$ or $\psi = 0$ denote truth,
while $\theta_\psi = 0$ corresponds to falsity.

For most purposes, it suffices to find a weak definition for a distal cell decomposition, as then we can let
$$\theta_\psi(y;\bar y) = \sup_{x,x'} \min(\psi(x;\bar y),\psi(x';\bar y), |\phi(x;y) - \phi(x';y)|\dot{-}\varepsilon),$$
and $\Theta = \{\theta_\psi: \psi \in \Psi\}$ will finish defining the distal cell decomposition.

We justify this definition of distal cell decompositions by showing that their existence is equivalent to distality.
First we show that distal cell decompositions follow from strong honest definitions, and then we will show that they imply distality, completing the cycle of equivalences.
\begin{lem}\label{lem_DCD}
Let $\phi(x;y)$ be a definable predicate such that $\phi(x;y)$ admits a strong honest definition.
Then $\phi(x;y)$ admits a distal cell decomposition for all $\varepsilon > 0$.
\end{lem}
\begin{proof}
    Let $\theta(x;z)$ be a strong honest definition for $\phi(x;y)$.
    Then by the density of formulas in definable predicates, let $\psi(x;z)$ be a formula which is always within $\frac{\varepsilon}{6}$ of $\frac{\varepsilon}{3} \dot{-} \theta(x;z)$.

    Fix $B \subseteq M^y$.
    Then for each $a \in M^x$, there is a tuple $d_a$ in $B^z$ such that $\theta(a;d_a) = 0$,
    and for all $a' \in M^x$ and $b \in B$, $|\phi(a;b) - \phi(a';b)| \leq \theta(a';d_a)$.
    Thus $|\psi(x;z) - \frac{\varepsilon}{3}| \leq \frac{\varepsilon}{6}$, so $\psi(a;d_a) \geq \frac{\varepsilon}{6} > 0$.
    If $a' \in M^x$ is such that $\psi(a';d_a) > 0$, then $\theta(a';d_a)<\frac{\varepsilon}{2}$, so for all $b \in B$, $|\phi(a;b) - \phi(a';b)|\leq \frac{\varepsilon}{2}$,
    and thus for all $a_1,a_2 \in M^x$ such that $\psi(a_1;d_a) > 0$ and $\psi(a_2;d_a) > 0$, we have $|\phi(a_1;d_a) - \phi(a_2;d_a)|\leq \varepsilon$.

    As $\psi(x;z)$ is a formula, it depends on only finitely many variables, so we may select $y_1,\dots,y_k$ to be copies of $y$ within $z$ including all variables on which $\psi$ depends.
    Then letting $\Psi = \{\psi(x;y_1,\dots,y_k)\}$, we check that $\Psi$ weakly defines a $\varepsilon$-distal cell decomposition.
    If $B_\psi$ is the set of all $\bar b$ such that some $d_a$ restricts to $\bar b$, we find that for all $a$, there is some $\bar b \in B_\psi$ such that $\psi(a;\bar b) > 0$,
    and for each $\bar b \in B_\psi$, and for all $a_1,a_2 \in M^x$ such that $\psi(a_1;\bar b) > 0$ and $\psi(a_2;\bar b) > 0$, we have $|\phi(a_1;\bar b) - \phi(a_2;\bar b)|\leq \varepsilon$.
\end{proof}

\begin{thm}\label{thm_SHD_to_distal}
If a metric theory $T$ is such that all formulas admit $\varepsilon$-distal cell decompositions for all $\varepsilon > 0$, then it is distal.
\end{thm}
\begin{proof}
    Fix $I + d + J$ indiscernible with indiscernible over $B$ and $I, J$ infinite. We will show that $I + d + J$ is indiscernible over $B$.
    To do this, let $a$ be a finite tuple from $A$.
    
    Let $\phi$ be a formula, and without loss of generality, assume $\phi(a;b_0,\dots,b_{2n}) = 0$ when $b_0 < \dots < b_{2n}$ is an increasing sequence in $I + J$.
    Fix $\varepsilon > 0$.
    We will show that for any $b_0 < \dots < b_{n-1} \in I, b_{n+1}<\dots < b_{2n} \in J$,
    $\phi(a;b_0,\dots,b_{n-1},d,b_{n+1},\dots,b_{2n})\leq \varepsilon$,
    implying that $I + d + J$ is $A$-indiscernible.
    
    Let $\Psi$ weakly define a $\varepsilon$-distal cell decomposition for $\phi(x;y_0,\dots,y_{2n})$.
    Fix a finite set $I_0 \subseteq I$ with $|I_0| \geq |z| + 2(2n+1)$.
    Then there is some $\psi(x;y_1,\dots,y_k) \in \Psi$ and some tuple $\bar b \in I^k$ such that $\psi(a,\bar b) > 0$ and
    for all $a'$ with $\psi(a';\bar b) > 0$, for all $\bar b' \in I^{2n + 1}$, $\phi(a;\bar b') \leq \varepsilon.$
    Thus
    $$\sup_x \max(\psi(x;\bar b),\varepsilon \dot{-} \phi(x;\bar b')) = 0.$$

    Because $I_0$ is large, there is an increasing sequence $b_0 <\dots < b_{2n}$ in $I_0$ disjoint from $\bar b$, and in particular, all of the elements in the sequence are either less than or greater than the entire tuple $\bar b$.
    
    Now let $b_0',\dots, b_{n-1}' \in I$, $b_{n+1}',\dots, b_{2n}' \in J$, and we will show that
    $$\phi(a;b_0',\dots,b_{n-1}',d,b_{n+1}',\dots,b_{2n}')\leq \varepsilon.$$
    There is some tuple $\bar b' \in I + J$ such that $(b_0',\dots, b_{n-1}',d,b_{n+1}',\dots, b_{2n}',\bar b')$ has the same order type as $b_0,\dots, b_{2n},\bar b$.
    By the indiscernibility of $I + d + J$, we find that
    \begin{align*}
        &\sup_x \max(\psi(x;\bar b'),\varepsilon \dot{-} \phi(x;b_0',\dots,b_{n-1}',d,b_{n+1}',\dots,b_{2n}')) \\
        = &\sup_x \max(\psi(x;\bar b),\varepsilon \dot{-} \phi(x;b_0,\dots,b_{2n}))\\
        =& 0
    \end{align*} 
    $$,$$
    and by the indiscernibility of $I + J$ over $A$, we have $\psi(a;\bar b') > 0$, so
    $$\phi(a;b_0',\dots,b_{n-1}',d,b_{n+1}',\dots,b_{2n}') \leq \varepsilon,$$ as desired.
    \end{proof}

\subsection{Reductions}
Having seen that all of these properties are equivalent to distality, we now provide some more ways to check whether a theory is distal.

We will show that the property of admitting strong honest definitions is closed under continuous combinations, which means that given quantifier elimination, it suffices to check that atomic formulas admit strong honest definitions.
\begin{lem}
    Let $\phi_1(x;y),\dots,\phi_n(x;y)$ be formulas that admit strong honest definitions.
    Let $u:[0,1]^n \to [0,1]$ be continuous.
    Then $$\phi(x;y) = u(\phi_1(x;y),\dots,\phi_n(x;y))$$ admits a strong honest definition.
\end{lem}
\begin{proof}
Define $F, G : ([0,1]^n \times [0,1]^n) \to [0,1]$ as follows, using the $\ell_\infty$-norm on $[0,1]^n$.
Let $F(a,a') = |a - a'|_{\ell_\infty}$ and
$G(a,a') = |u(a) - u(a')|$.
As $u$ is continuous between two compact metric spaces, it is uniformly continuous,
so for each $\varepsilon > 0$, there is some $\delta > 0$ such that
$|a - a'|_{\ell_\infty} \leq \delta$ implies $|u(a) - u(a')|\leq \varepsilon$.
Thus by \cite[Proposition 2.10]{mtfms}, there is some increasing continuous $\alpha : [0,1] \to [0,1]$ such that $\alpha(0)=0$ and
$\forall a,a', G(a,a')\leq \alpha(F(a,a'))$.

Now for $1 \leq i \leq n$, let $\theta_i(x;z)$ be a strong honest definition for $\phi_i(x;y)$.
Then let $\theta(x;z_1,\dots,z_n) = \alpha\left(\min_{1 \leq i \leq n}\theta_i(x;z_i)\right)$.
We check that $\theta$ is a strong honest definition for $\phi$.

Let $M \preceq \mathcal{U}$, $A$ closed in $M^y$, $a \in M^x$, $(M,A) \preceq (M',A')$ be sufficiently saturated.
Let $d_1,\dots,d_n \in A'$ be such that for each $i$, $\theta_i(x;d_i)$ is a strong honest definition for $\phi_i(a;y)$ over $A$.
Then we see that 
$$\theta(a;d) = \alpha\left(\max_{1 \leq i \leq n}\theta_i(a;z_i)\right) = 0.$$
Now let $a' \in M'^x$, $b \in A$. We see that
\begin{align*}
    |\phi(a';b) - \phi(a;b)| &= G(\phi_1(a;b),\dots, \phi_n(a;b),\phi_1(a';b),\dots,\phi_n(a';b))\\
&\leq \alpha(F(\phi_1(a;b),\dots, \phi_n(a;b),\phi_1(a';b),\dots,\phi_n(a';b)))\\
&= \alpha\left(\max_{1 \leq i \leq n} |\phi_i(a';b) - \phi_i(a;b)|\right)\\
&\leq \alpha\left(\max_{1 \leq i \leq n}\theta_i(a';d_i)\right)\\
&= \theta(a';d).
\end{align*}
\end{proof}

\begin{cor}\label{cor_qe}
    As a corollary, we see that if $T$ eliminates quantifiers and all atomic formulas admit strong honest definitions, then $T$ is distal.
\end{cor}

We can also reduce to one variable.
\begin{thm}\label{thm_one_var}
    Let $T$ be an NIP theory. Then $T$ is distal if and only if any of the following equivalent conditions hold:
    \begin{itemize}
        \item Any indiscernible $I + d + J$ with $I + J$ indiscernible over a singleton $b$ is indiscernible over $b$
        \item For any $A \subset \mathcal{M}$, global $A$-invariant type $p$ and singleton $b$, if $I \vDash p^{(\omega)}|_{Ab}$, then $p|_{AI}$ and $\mathrm{tp}(b/AI)$ are weakly orthogonal
        \item Any predicate $\phi(x;y)$ with $|x| = 1$ admits a strong honest definition
        \item Any predicate $\phi(x;y)$ with $|x| = 1$ admits an $\varepsilon$-distal cell decomposition for every $\varepsilon > 0$.
    \end{itemize}
\end{thm}
\begin{proof}
    Clearly distality implies all of these conditions.

    These conditions are all equivalent by following the proofs of the implications in \ref{thm_distal} and keeping track of the length of tuples.
    We will prove that the indiscernible condition implies distality and the strong honest definition condition implies distality.
    The first proof is more straightforward, but we will also construct explicit strong honest definitions for predicates with more variables, generalizing the constructions in \cite[Theorem 3.1]{andersonF} and \cite[Proposition 1.9]{distal_val}.

    First we show that if distality fails, the first condition fails.
    Let $I + d + J$ be indiscernible, and let $b$ be a tuple such that $I + J$ is indiscernible over $b$, but $I + d + J$ is not indiscernible over $b$. 
    Then $I + d + J$ it is not indiscernible over some finite subtuple of $b$, and we may assume $b$ is finite.
    Let $n$ be minimal such that there exists $b = (b_1,\dots,b_n)$ satisfying these properties.
    
    For a sequence $S$ and a tuple $b'$, let $S^\frown b'$ be the tuple obtained by concatenating $b'$ to each term of $S$.
    Then $S$ is indiscernible over $b'$ if and only if $S^\frown b'$ is indiscernible.

    We know that $I + d + J$ is indiscernible over $(b_1,\dots,b_{n-1})$, so $(I + d + J)^\frown(b_1,\dots,b_{n-1})$ is indiscernible,
    and $(I + J)^\frown(b_1,\dots,b_{n-1})$ is indiscernible over $b_n$, but $(I + d + J)^\frown(b_1,\dots,b_{n-1})$ is not indiscernible over $b_n$,
    so this sequence fails the first criterion over the singleton $b_n$.

    Now we provide an explicit construction of strong honest definitions.
    Let $T$ be a theory in which every definable predicate $\phi(x;y)$ with $|x| = 1$ admits a strong honest definition.
    To show that every definable predicate $\phi(x;y)$ admits a strong honest definition, it suffices to show it for all predicates with $|x|$ finite, as every predicate is a uniform limit of such predicates, and by Lemma \ref{lem_SHD_lim}, uniform limits of predicates with strong honest definitions have strong honest definitions.

    Assume for induction that this holds for every definable predicate with $|x| \leq n$, and let $\phi(x_0,x;y)$ be a definable predicate with $|x| = n$.
    We will now repartition the variables of $\phi$ several ways, and find strong($^*$) honest definitions for each repartition.
    Then by assumption, there exists a strong honest definition $\theta_0(x_0;z_0)$ for $\phi(x_0;x,y)$.
    As $z_0$ is a (possibly countable) tuple of copies of $(x,y)$, and we will be interested in considering $\theta_0$ as a strong honest definition over sets of the form $\{a\} \times A$ for $A \subseteq M^y$,
    we will assume that each copy of $x$ is equal, and write the predicate as $\theta_0(x_0;x,z_0)$, where $z_0$ is a tuple of copies of $y$.
    Then we let $\psi^+(x;y,z_0) = \sup_{x_0} \phi(x_0,x;y)\dot{-}\theta_0(x_0;x,z_0)$, and $\psi^-(x;y,z_0) = 1 - \sup_{x_0} (1 - \phi(x_0,x;y))\dot{-}\theta_0(x_0;x,z_0)$.
    As $|x| = n$, there are also strong honest definitions $\theta^+(x;z_+), \theta^-(x;z_-)$ for $\psi^+(x;y,z_0), \psi^-(x;y,z_0)$ respectively.

    We claim that $\theta(x_0,x;z_0,z_+,z_-) = \theta_0(x_0;x,z_0) + \theta^+(x;z_+) + \theta^-(x;z_-)$ is a strong honest definition for $\phi(x_0,x;y)$.
    Now fix $A \subseteq M^y$, $a_0 \in M, a \in M^x$.
    Let $d_0$ be such that $\theta_0(x_0;a,d_0)$ is a strong honest definition for $\phi(a_0;x,y)$ over $\{a\}\times A$,
    and let $d_\pm$ be such that $\theta^\pm(x;d_\pm)$ is a strong honest definition for $\psi^\pm(a;y,z_0)$ over $A \times \{d_0\}$.
    By definition, we will have $\theta(a_0,a;d_0,d_+',d_-') = 0 + 0 + 0$.
    For any $a_0' \in M$, as $\theta_0(a_0';a,d_0) \geq |\phi(a_0',a;b) - \phi(a_0,a;b)|$ and thus $\phi(a_0',a;b) \leq \phi(a_0,a;b) + \theta_0(a_0';a,d_0)$,
        we have $\psi_+(a;b,d_0) = \sup_{x_0}\phi(a_0',a;b)\dot{-}\theta_0(a_0';a,d_0) \leq \phi(a_0,a;b)$, and by a similar calculation,
        $\psi_-(a;b,d_0) \geq \phi(a_0,a;b)$.

    Now let $a_0' \in M$, $a' \in M^n$, and $b \in A$, and we will show that $|\phi(a_0,a;b) - \phi(a_0',a';b)|\leq \theta(a_0',a';d)$.
    First we will show that $\phi(a_0',a';b) \leq \phi(a_0,a;b) + \theta(a_0',a';d)$.

    We see that $|\psi^+(a';b,d_0) - \psi^+(a;b,d_0)| \leq \theta^+(a';d_+)$, so 
    \begin{align*}
        \phi(a_0',a';b)\dot{-}\theta_0(a_0';a',d_0)
        &\leq \psi^+(a';b,d_0)\\
        &\leq \psi^+(a;b,d_0) + \theta^+(a';d_+)\\
        & \leq \phi(a_0,a;b) + \theta^+(a';d_+)
    \end{align*}
    and thus 
    \begin{align*}
        \phi(a_0',a';b) & \leq \phi(a_0,a;b) + \theta_0(a_0';a',d_0) + \theta^+(a';d_+)\\
        & \leq \phi(a_0,a;b) + \theta(a_0';a',d_0,d_+,d_-).
    \end{align*}
    By similar logic, 
    \begin{align*}
        \phi(a_0',a';b) & \geq \phi(a_0,a;b) - \theta_0(a_0';a',d_0) - \theta^-(a';d_-)\\
        & \geq \phi(a_0,a;b) - \theta(a_0';a',d_0,d_+,d_-).
    \end{align*}
\end{proof}

\bibliography{ref.bib}

\providecommand{\MR}{\relax\ifhmode\unskip\space\fi MR }
\providecommand{\MRhref}[2]{%
  \href{http://www.ams.org/mathscinet-getitem?mr=#1}{#2}
}
\providecommand{\href}[2]{#2}
\begin{thebibliography}{10}

\bibitem{alon97}
Noga Alon, Shai Ben-David, Nicolo Cesa-Bianchi, and David Haussler,
  \textsl{Scale-sensitive dimensions, uniform convergence, and learnability},
  Journal of the ACM (JACM) \textbf{44} (1997), no.~4, 615--631.

\bibitem{PCC}
Noga Alon, Steve Hanneke, Ron Holzman, and Shay Moran, \textsl{A theory of
  {PAC} learnability of partial concept classes}, 2021 IEEE 62nd Annual
  Symposium on Foundations of Computer Science (FOCS), IEEE, 2022,
  pp.~658--671.

\bibitem{alon_kleitman}
Noga Alon and Daniel~J Kleitman, \textsl{Piercing convex sets and the
  {H}adwiger-{D}ebrunner $(p, q)$-problem}, Advances in Mathematics \textbf{96}
  (1992), no.~1, 103--112.

\bibitem{anderson2}
Aaron Anderson, \textsl{Generically stable measures and distal regularity in
  continuous logic}, 2023.

\bibitem{anderson3}
Aaron Anderson and Ita\"i Ben~Yaacov, \textsl{Examples and nonexamples of
  distal metric structures}, forthcoming.

\bibitem{distal_val}
Matthias Aschenbrenner, Artem Chernikov, Allen Gehret, and Martin Ziegler,
  \textsl{Distality in valued fields and related structures}, Transactions of
  the American Mathematical Society \textbf{375} (2022), no.~7, 4641--4710.

\bibitem{randomVC}
Ita\"i Ben~Yaacov, \textsl{Continuous and random {Vapnik-Chervonenkis}
  classes}, Israel Journal of Mathematics \textbf{173} (2009), no.~1, 309--333.

\bibitem{StabGrps}
Ita{\"\i} Ben~Yaacov, \textsl{Stability and stable groups in continuous logic},
  The Journal of Symbolic Logic \textbf{75} (2010), no.~3, 1111--1136.

\bibitem{mtfms}
Ita\"i Ben~Yaacov, Alexander Berenstein, C.~Ward Henson, and Alexander
  Usvyatsov, \textsl{Model theory for metric structures}, London Mathematical
  Society Lecture Note Series \textbf{350} (2008), 315.

\bibitem{byt}
Ita{\"\i} Ben~Yaacov and Todor Tsankov, \textsl{Weakly almost periodic
  functions, model-theoretic stability, and minimality of topological groups},
  Transactions of the American Mathematical Society \textbf{368} (2016),
  no.~11, 8267--8294.

\bibitem{local_stab}
Ita{\"\i} Ben~Yaacov and Alexander Usvyatsov, \textsl{Continuous first order
  logic and local stability}, Transactions of the American Mathematical Society
  \textbf{362} (2010), no.~10, 5213--5259.

\bibitem{csr}
Nicolas Chavarria, Gabriel Conant and Anand Pillay, \textsl{Continuous stable
  regularity}, 2021.

\bibitem{cgs}
Artem Chernikov, David Galvin and Sergei Starchenko, \textsl{Cutting lemma and
  {Z}arankiewicz's problem in distal structures}, Selecta Mathematica
  \textbf{26} (2020), 1--27.

\bibitem{cs1}
Artem Chernikov and Pierre Simon, \textsl{Externally definable sets and
  dependent pairs}, Israel Journal of Mathematics \textbf{194} (2013), no.~1,
  409--425.

\bibitem{cs2}
Artem Chernikov and Pierre Simon, \textsl{Externally definable sets and
  dependent pairs {II}}, Transactions of the American Mathematical Society
  \textbf{367} (2015), no.~7, 5217--5235.

\bibitem{distal_reg}
Artem Chernikov and Sergei Starchenko, \textsl{Regularity lemma for distal
  structures}, Journal of the European Mathematical Society \textbf{20} (2018),
  no.~10, 2437--2466.

\bibitem{ct}
Artem Chernikov and Henry Towsner, \textsl{Hypergraph regularity and higher
  arity {VC}-dimension}, 2020.

\bibitem{udtfs}
Shlomo Eshel and Itay Kaplan, \textsl{On uniform definability of types over
  finite sets for {NIP} formulas}, Journal of Mathematical Logic \textbf{21}
  (2021), no.~03, 2150015.

\bibitem{ibarlucia1}
Tom{\'a}s Ibarluc{\'\i}a, \textsl{The dynamical hierarchy for {R}oelcke
  precompact {P}olish groups}, Israel Journal of Mathematics \textbf{215}
  (2016), no.~2, 965--1009.

\bibitem{kp21}
Krzysztof Krupi{\'n}ski and Adri{\'a}n Portillo, \textsl{On stable quotients},
  Notre Dame Journal of Formal Logic \textbf{63} (2022), no.~3, 373--394.

\bibitem{matousek_GTM}
Ji\v{r}{\'\i} Matou\v{s}ek, \textsl{Lectures on discrete geometry},
  Springer-Verlag, Berlin, Heidelberg, 2002.

\bibitem{matousek_helly}
Ji\v{r}{\'{\i}} Matou\v{s}ek, \textsl{Bounded {VC}-dimension implies a
  fractional {H}elly theorem}, Discrete \& Computational Geometry \textbf{31}
  (2004), 251--255.

\bibitem{melleray}
Julien Melleray, \textsl{A note on {H}jorth's oscillation theorem}, The Journal
  of Symbolic Logic \textbf{75} (2010), no.~4, 1359--1365.

\bibitem{convexFS}
Shahar Mendelson, \textsl{On the size of convex hulls of small sets}, J. Mach.
  Learn. Res. \textbf{2} (2002), 1–18.

\bibitem{rudin}
Walter Rudin, \textsl{Real and complex analysis, 3rd ed.}, McGraw-Hill, Inc.,
  USA, 1987.

\bibitem{distal_simon}
Pierre Simon, \textsl{Distal and non-distal {NIP} theories}, Annals of Pure and
  Applied Logic \textbf{164} (2013), no.~3, 294--318.

\bibitem{nip_guide}
Pierre Simon, \textsl{A guide to {NIP} theories}, Cambridge University Press,
  2015.

\bibitem{simon_distal_reg}
Pierre Simon, \textsl{A note on ``{R}egularity lemma for distal structures''},
  Proceedings of the American Mathematical Society \textbf{144} (2016), no.~8,
  3573--3578.

\bibitem{wainwright}
Martin~J Wainwright, \textsl{High-dimensional statistics: A non-asymptotic
  viewpoint}, vol.~48, Cambridge university press, 2019.

\end{thebibliography}
\bibliographystyle{ijmart}

\end{document}